\documentclass[article]{siamart190516}

\usepackage{amsmath,amssymb}
\usepackage[utf8]{inputenc}
\usepackage{mathtools}
\usepackage{color,graphicx}
\usepackage{thmtools, thm-restate}
\usepackage{enumitem}
\usepackage{array,multirow}

\reversemarginpar

\usepackage{algorithm}
\usepackage{algpseudocode}

\newlist{myitemize}{itemize}{1}
\setlist[myitemize,1]{leftmargin=0.2in}

\newlist{myitemize2}{itemize}{1}
\setlist[myitemize2,1]{leftmargin=0.4in}

\usepackage{tikz}
\usetikzlibrary{patterns}
\usetikzlibrary{arrows}
\usetikzlibrary{3d,calc}
\usepackage{etoolbox}
\makeatletter

\patchcmd{\@addmarginpar}{\ifodd\c@page}{\ifodd\c@page\@tempcnta\m@ne}{}{}
\providecommand*{\cupdot}{%
  \mathbin{%
    \mathpalette\@cupdot{}%
  }%
}
\newcommand*{\@cupdot}[2]{%
  \ooalign{%
    $\m@th#1\cup$\cr
    \hidewidth$\m@th#1\cdot$\hidewidth
  }%
}
\providecommand*{\bigcupdot}{%
  \mathop{%
    \vphantom{\bigcup}%
    \mathpalette\@bigcupdot{}%
  }%
}
\newcommand*{\@bigcupdot}[2]{%
  \ooalign{%
    $\m@th#1\bigcup$\cr
    \sbox0{$#1\bigcup$}%
    \dimen@=\ht0 %
    \advance\dimen@ by -\dp0 %
    \sbox0{\scalebox{2}{$\m@th#1\cdot$}}%
    \advance\dimen@ by -\ht0 %
    \dimen@=.5\dimen@
    \hidewidth\raise\dimen@\box0\hidewidth
  }%
}
\makeatother
\reversemarginpar


\newcommand{\cF}{\mathcal{F}}

\newcommand{\cH}{\mathcal{H}}

\newcommand{\cNl}{\mathcal{N}_\textrm{l}}
\newcommand{\cNlslow}{\mathcal{N}_\textrm{l}^\text{S}}
\newcommand{\cNlfast}{\mathcal{N}_\textrm{l}^\text{F}}

\newcommand{\cO}{\mathcal{O}}

\newcommand{\cTl}{\mathcal{T}_\textrm{l}}
\newcommand{\cTlslow}{\mathcal{T}_\textrm{l}^\text{S}}
\newcommand{\cTlfast}{\mathcal{T}_\textrm{l}^\text{F}}
\newcommand{\cTg}{\mathcal{T}_\textrm{g}}
\newcommand{\cTgslow}{\mathcal{T}_\textrm{g}^\text{S}}
\newcommand{\cTgfast}{\mathcal{T}_\textrm{g}^\text{F}}
\newcommand{\cM}{\mathcal{M}}

\newcommand{\cQ}{\mathcal{Q}}

\newcommand{\cX}{\mathcal{X}}

\newcommand{\beq}{\begin{equation}}
\newcommand{\eeq}{\end{equation}}
\def\bals#1\eals{\begin{align*} #1 \end{align*}}
\def\bal#1\eal{\begin{align} #1 \end{align}}

\newcommand\Dom{\Omega}

\newcommand\RR{\mathbb{R}}

\newcommand\bbS{\mathbb{S}}
\newcommand\TT{\mathbb{T}}

\newcommand\NN{\mathbb{N}}

\newcommand\VV{\mathbb{V}}

\newcommand\Htrain{{\cH_\textrm{g}}}
\newcommand\Ntrain{{N_\textrm{g}}}

\newcommand\Lap\Delta

\newcommand\abs[1]{\left\lvert #1 \right\rvert}

\newcommand\dx{\,\mathrm{d}x}

\def\bpde#1\epde{\[\left\{\begin{aligned}#1\end{aligned}\right. \]}
\def\inbpde#1\inepde{\left\{\begin{aligned}#1\end{aligned}\right.}
\def\binpde#1\einpde{\left\{\begin{aligned}#1\end{aligned}\right.}

\newcommand\Norm[2]{\left\lVert { #1 } \right\rVert_{#2}}

\def\cD{\mathcal{D}}

\def\cV{\mathcal{V}}

\def\half{\frac{1}{2}}

\def\bv{\mathbf{v}}

\def\bfC{\mathbf{C}}

\def\bfI{\mathbf{I}}

\def\bfT{\mathbf{T}}

\def\bfI{\mathbf{I}}

\def\b0{\mathbf{0}}

\def\bv{\mathbf{v}}

\def\bmu{\mbox{\boldmath$\mu$}}

\def\bfeta{t,\boldsymbol{\mu}}
\newcommand\bfetai[1]{t_{#1},\boldsymbol{\mu}_{#1}}

\def\eps{\varepsilon}

\def\bbmat{\begin{bmatrix}[r]}
\def\ebmat{\end{bmatrix}}

\newcommand{\barr}{\begin{array}}
\newcommand{\ea}{\end{array}}
\newcommand{\bea}{\begin{eqnarray}}
\newcommand{\eea}{\end{eqnarray}}
\newcommand{\bt}{\begin{table}}
\newcommand{\et}{\end{table}}

\DeclareMathOperator\Id{Id}
\DeclareMathOperator\Span{span}

\DeclareMathOperator\argmax{argmax}
\DeclareMathOperator\supp{supp}

\DeclareMathOperator\sign{sign}

\DeclareMathOperator\setint{int}

\theoremstyle{plain}

\newsiamthm{cond}{Condition}
\newsiamremark{remark}{Remark} 
\newsiamremark{example}{Example}

\numberwithin{equation}{section}

\newcommand\fixsp[1]{{\tt {\footnotesize #1}}}

\newcommand\uh{u_\delta}

\newcommand\urb{u_\textrm{rb}}

\newcommand\cMh{\mathcal{M}_\delta}
\newcommand\VVh{\mathbb{V}_\delta}

\newcommand\TTrb{\mathbb{T}_\textrm{rb}}
\newcommand\VVrb{\mathbb{V}_\textrm{rb}}

\newcommand\hT{\hat{T}}
\newcommand\hTpb{\hat{T}_{(\bfeta)}}
\newcommand\tfin{t_\textrm{F}}

\newcommand\Nl{{N_{\textrm{l}}}}
\newcommand\Ng{{N_{\textrm{g}}}}
\newcommand\cHl{\mathcal{H}_{\textrm{l}}} 
\newcommand\cHlslow{\mathcal{H}_{\textrm{l}}^\text{S}} 
\newcommand\cHlfast{\mathcal{H}_{\textrm{l}}^\text{F}} 

\newcommand\cHg{\mathcal{H}_{\textrm{g}}}
\newcommand\cHgslow{\mathcal{H}_{\textrm{g}}^\text{S}} 
\newcommand\cHgfast{\mathcal{H}_{\textrm{g}}^\text{F}}

\newcommand\cSl{\mathcal{S}_{\textrm{l}}}
\newcommand\cSlslow{\mathcal{S}_{\textrm{l}}^\text{S}}
\newcommand\cSlfast{\mathcal{S}_{\textrm{l}}^\text{F}}
\newcommand\cSg{\mathcal{S}_{\textrm{g}}}
\newcommand\cSgslow{\mathcal{S}_{\textrm{g}}^\text{S}}
\newcommand\cSgfast{\mathcal{S}_{\textrm{g}}^\text{F}}

\definecolor{darkgreen}{rgb}{0.01, 0.75, 0.24}%

\newcommand{\TheTitle}{Manifold Approximations via Transported Subspaces:\\ 
model reduction for transport-dominated problems}

\newcommand{\remove}[1]{\textcolor{red}{[Text removed.]}}%

\begin{document}

\ifpdf
\DeclareGraphicsExtensions{.pdf, .jpg, .tif}
\else
\DeclareGraphicsExtensions{.eps, .jpg}
\fi

\title{\TheTitle}

\author{Donsub Rim%
  \thanks{Courant Institute of Mathematical Sciences, %
  New York University, New York, NY 10012 %
  ({\tt dr1653@nyu.edu}, {\tt pehersto@cims.nyu.edu}).}%
  \and%
  Benjamin Peherstorfer\footnotemark[1]
  \and%
  Kyle T. Mandli%
  \thanks{Department of Applied Physics and Applied Mathematics, %
  Columbia University, New York, NY 10027 %
  ({\tt kyle.mandli@columbia.edu}).}%
}
\maketitle

\begin{abstract} 
This work presents a method for constructing online-efficient reduced models of
large-scale systems governed by parametrized nonlinear scalar conservation
laws. The solution manifolds induced by transport-dominated problems such as
hyperbolic conservation laws typically exhibit nonlinear structures, which
means that traditional model reduction methods based on linear approximations
are inefficient when applied to these problems. In contrast, the approach
introduced in this work derives reduced approximations that are nonlinear by
explicitly composing global transport dynamics with locally linear
approximations of the solution manifolds. A time-stepping scheme evolves the
nonlinear reduced models by transporting local approximation spaces along the
characteristic curves of the governing equations. The proposed computational
procedure allows an offline/online decomposition and is online-efficient in
the sense that the complexity of accurately time-stepping the nonlinear reduced
model is independent of that of the full model. Numerical experiments with
transport through heterogeneous media and the Burgers' equation show orders of
magnitude speedups of the proposed nonlinear reduced models based on
transported subspaces compared to traditional linear reduced models and full
models.
\end{abstract}


\begin{AMS}
\texttt{78M34}, \texttt{41A46}, \texttt{35F20}, \texttt{78M12}
\end{AMS}

\section{Introduction}
Model reduction for transport-dominated problems can be challenging because the
corresponding solution manifolds typically exhibit high dimensional features
and so can only be well-approximated by subspaces with high dimension; that is,
the Kolmogorov $N$-widths of the solution manifolds decay slowly
\cite{Ohlberger16,Greif19}.  To overcome this challenge of slowly decaying
Kolmogorov $N$-widths, this work introduces transported subspaces, which are
subspaces conjoined with nonlinear transportation.  The corresponding
approximations are nonlinear and so can lead to efficient reduced models even
if the Kolmogorov $N$-widths of the solution manifolds decay slowly. 

The difficulty in developing reduced models for transport-dominated problems
was recognized in \cite{rowley00}, which also introduced a template-fitting
approach. Subsequent works to overcome this problem fall broadly into two
groups. The first group attempts to find relevant nonlinear transformations of
the solution manifolds. These include the method of freezing
\cite{Beyn04,Ohlberger13}, or approximated Lax-Pairs \cite{Gerbeau14}. The work
\cite{iollo14} made a connection to the optimal transport problem in finding a
nonlinear low-dimensional structure. More works in this group include shock
curve estimation \cite{taddei14}, shifted proper orthogonal decomposition
(sPOD) \cite{schulze18}, Lagrangian basis method \cite{Mojgani17}, greedy
generalization of template-fitting \cite{rim17reversal}, transformed snapshot
interpolation \cite{welper17,welper19}, a machine-learning approach based on
autoencoders \cite{lee19}, characteristic dynamic mode decomposition
\cite{sesterhenn2019}, registration methods \cite{taddei19}, Wasserstein
barycenters \cite{ehrlacher19}, calibrated manifolds
\cite{Cagniart2019,nonino19}, unsupervised traveling wave identification with
shifting truncation \cite{mendible19}, and a generalization of the moving
finite element method \cite{black19}. For nonlinear hyperbolic conservation
laws \cite{fvmbook}, there is a relationship between optimal transport and
scalar conservation laws \cite{bolley05}, and it was demonstrated in
\cite{rim18a,rim18b,rim19b} that a nonlinear interpolation procedure called
displacement interpolation by pieces (DIP) is able to capture the nonlinear
transport structure, which leads to localized reduced models. 
  
The second group focuses on online adaptive methods that update local reduced
spaces depending on time and parameters. For example, the work
\cite{carlberg15} proposes an adaptive refinement of reduced bases similar to
$h$-adaptive finite-element approaches. In \cite{pehersto15adeim,P18AADEIM},
the locality of coherent structures that are transported through the spatial
domain is exploited to adapt spaces from only a few samples. While these
adaptive procedures achieve speedups compared to traditional reduced models,
the corresponding computational methods have costs that scale with the
complexity of the full-model discretization. Additional relevant works
include dynamical low-rank approximation \cite{sapsis09} and dynamically
orthogonal (DO) schemes \cite{koch07}, the use of $L^1$-norm minimization
\cite{amsallem} and the Petrov-Galerkin space-time approach \cite{yano14}.

The model reduction method proposed in this work is called Manifold
Approximations via Transported Subspaces (MATS) and merges ideas from both
groups mentioned above. MATS produces an adaptive scheme that uses nonlinear
transformations to generate an adaptive reduced basis during the online stage.
Although the effect of the transformation is nonlinear, the transformation
itself is built from a linear combination of low-rank transport modes. This
specific structure enables the derivation of  the time-dependent adaptation
directly from the governing equations. Combining the MATS procedure with a
projection-based time update results in online-efficient reduced models, which
means that the cost complexity the reduced models at each time-step is
independent of the number of degrees of freedom of the full-model
discretizations.

Three main ideas are introduced in this work:
\begin{itemize}
    \item[(1)] the transported subspace generated by low-rank transport modes that are explicit to derive and compute, which provide a low-rank approximation to the transport dynamics necessary to approximate the solution manifold,
    \item[(2)] interpolation particles that enable online-efficient adaptations of transported subspaces,
    \item[(3)] an online-efficient time-stepping scheme for the proposed reduced model.
\end{itemize}

These ideas share commonalities with those in existing works: Various forms
of low-rank transformation of the domain were considered in
\cite{iollo14,rim18a,taddei19}; The idea of using a transformation of the
domain along with an interpolation method is well-known and standard
\cite{crb-book}. However, we modify and combine these ideas for efficient
computations during the online stage: We are not aware of previous methods that
achieve the online complexity and accuracy for the problems we consider here.
In this work, we restrict our scope to the parametrized classical solutions of
non-homogeneous nonlinear scalar conservation laws with constant influx
boundary condition. 

This manuscript is organized as follows. In \cref{sec:prelim}, we formulate our
problem setup and set forth basic notations. In \cref{sec:tsub}, we introduce a
notion of a Kolmogorov width for transported subspaces. In \cref{sec:offline},
we provide a concrete procedure for constructing the transported subspaces. In
\cref{sec:online}, we construct the reduced model by building on the notion of
interpolation particles and the change of basis formula. An algorithmic summary
is given in \cref{sec:alg}. In \cref{sec:numerics}, we provide numerical
examples that illustrate the speedups obtained with our method. The
extensibility of the methods introduced here to more general settings is
discussed in \cref{sec:conclusion}.

\clearpage

\section{Problem formulation and preliminaries}
\label{sec:prelim}

In this section, we describe the parametrized partial differential equations (PDE) under consideration and introduce notations and definitions to be used throughout the work.

\subsection{Parametrized scalar conservation laws}

For the spatial domain $\Dom := (x_\ell,x_r) \subset \RR$ with $|x_\ell|, |x_r| < \infty$, and for parameters $\bmu \in \cD \subset \RR^d$, we seek a real-valued solution $u \in C^1(\Dom \times (0,\tfin) \times \cD)$ with $|\tfin| < \infty$ that satisfies
    \begin{equation}
        \left\{ \begin{aligned}
           \partial_t u + \partial_x [f(u,x;\bmu)]
           &= 
           \psi(u,x;\bmu), 
            \qquad 
            (x,t) \in \Dom \times (0,\tfin) ,\\
            u(x,0;\bmu) &= u_0(x), \\
            u(x_\ell,t;\bmu) &= u_0(x_\ell),
        \end{aligned} \right.
    \label{eq:prob}
    \end{equation}
in which $f(\cdot , \cdot ; \bmu) \in C^\infty(\RR^2)$ is a convex function in
the first variable and $\psi(\cdot, \cdot \, ;\bmu) \in C^\infty(\RR^2)$. The
initial and boundary values at $x = x_\ell$ are independent of time and
parameter. For a fixed parameter $\bmu$, the initial boundary value problem
\eqref{eq:prob} has been studied extensively both analytically and numerically
\cite{lax73claw,serre99,fvmbook}. In this work, we restrict our attention
to the classical solutions in $C^1$.

\subsection{Reduced approximation}

Let us denote the time-parameter domain by $\cH := (0,\tfin) \times \cD$. For all $(t,\bmu) \in \cH$, the solution $u(\cdot,t;\bmu)$ of \cref{eq:prob} is in the vector space $\VV := C^0(\Dom)$ equipped with the inner product $(f,g) := \int_\Dom fg \dx$ for $f,g \in \VV$. The inner product induces the norm $\Norm{\cdot}{\VV}$. Let us denote by $u_\delta \in  C^0(\Dom \times (0,\tfin) \times \cD)$ a continuous piecewise polynomial approximation of $u$, to which we refer to as the full-model solution in the following. Let $\VVh \subset \VV$ be the $N_\delta$-dimensional space of continuous piecewise polynomial functions defined on a grid $x_\ell < x_\ell + \Delta x < ... < x_\ell + (N_\delta-1) (\Delta x) = x_r$ of uniform width $\Delta x = \delta$ (with $(N_\delta-1) \delta = x_r - x_\ell$). Let $\{\varphi_n\}_{n = 1}^{N_\delta}$ be a basis of $\VVh$, then suppose we can represent $u_\delta$ as a linear combination of this basis,
\begin{equation}
    \uh(x,t;\bmu) = \sum_{n=1}^{N_\delta} b_n(\bfeta) \varphi_n(x),
    \label{eq:fom_repr}
\end{equation}
with coefficients $\{b_n(\bfeta)\}_{n=1}^{N_\delta}$. We assume that $u_\delta$ is a uniformly accurate approximation of the classical solution $u$, i.e., there exists a constant $C > 0$ such that for all $(t,\bmu) \in \cH$, the bound
\begin{equation}
    \Norm{u(\cdot,t;\bmu) - u_\delta(\cdot,t;\bmu)}{\VV} \le C \delta^r
    \label{eq:fm_error}
\end{equation}
holds for some $r \ge 1$. The solution manifold $\cMh$ corresponding to the full-model solutions of \cref{eq:prob} is
\beq
    \cMh := \{\uh(\cdot,t;\bmu) : (\bfeta) \in \cH \} \subset \VVh.
\eeq

The Kolmogorov $N$-width \cite{pinkus12} of the solution manifold $\cM_\delta$ is 
\begin{equation}
    d(N; \cMh) = 
        \inf_{\substack{\VVrb \subset \VV \\ \dim(\VVrb) = N}} 
        \sup_{\uh \in \cMh} \inf_{v \in \VVrb} 
        \Norm{\uh - v}{\VV}\,,
    \label{eq:nwidth}
\end{equation}
and gives the best possible error of approximating all functions in $\cMh$ in a
subspace $\VVrb \subset \VV$ of dimension $N = \dim \VVrb$. Traditional
model reduction aims to construct a basis $\{\zeta_n\}_{n=1}^N$ of a subspace
$\VVrb$ with which the solution in $\cMh$ can be approximated well. The reduced
approximation is a linear combination of the basis $\{\zeta_n\}_{n=1}^N$,
\begin{equation}
    \urb(x,t;\bmu) = \sum_{n=1}^N \beta_n(\bfeta) \zeta_n(x)\,,
    \label{eq:rb_repr}
\end{equation}
with coefficients $\{\beta_n(\bfeta)\}_{n = 1}^{N}$; see, e.g., \cite{crb-book,siamrev-survey} and references therein.

For problems of hyperbolic or transport-dominated types, the Kolmogorov
$N$-width of $\cMh$ for the problem \cref{eq:prob} can decay slowly
\cite{Ohlberger13,Greif19}, which implies that traditional model reduction
fails: to obtain an acceptable accuracy with a reduced approximation of the
form \cref{eq:rb_repr}, the number of terms $N$, and thus the dimension of
$\VVrb$, must be large. Note that, in case the system \cref{eq:prob} is a
linear time-invariant system, the $N$-width is equivalent to the Hankel
singular values \cite{Unger19}.

\section{Transported subspaces}\label{sec:tsub}

To overcome the limitation of reduced approximations of the form \cref{eq:rb_repr}, we introduce approximations that adaptively vary the subspace $\VVrb$ by using a different subspace $\VVrb(\bfeta)$ depending on $(\bfeta) \in \cH$. This section proposes a specific structure for constructing these adaptive subspaces. It builds on what we will call finite-dimensional subsets of transport maps, which lead to adaptive subspaces of low dimension and, at the same time, guarantee efficient adaptations with respect to time and parameters.  

Furthermore, we provide an intuition as to why our construction is expected to produce an accurate approximation for solution manifolds of transport-dominated problems. We do so by providing a generalization of the notion of the Kolmogorov $N$-width \cref{eq:nwidth} which we call the $(N,M)$-width.  The new argument $M$ corresponds to the dimension of the finite-dimensional subset of transport maps. We give examples of transport-dominated problems for which the $(N,M)$-width decays rapidly and the Kolmogorov $N$-width decays slowly.

\subsection{Finite-dimensional subsets of transport maps} 
\label{sec:tmap}

We describe transport maps via diffeomorphisms over $\RR$. Throughout, we define a diffeomorphism as a piecewise continuously differentiable homeomorphism of $\RR$.

\begin{definition}
Let $\TT$ denote the subgroup of diffeomorphisms on $\RR$ that are increasing (strictly monotone).
\end{definition}

Function compositions (pullbacks) will play a key role. We will denote by $\flat$ the composition of $\xi: \RR \to \RR$ with the inverse of $T \in \TT$,
\beq
    T^\flat \xi(x) := \xi(T^{-1}(x)).
    \label{eq:pullback}
\eeq

Next we define the composition with the functions defined on the domain $\Dom$ by extending them to the real line. 

\begin{definition}\label{def:transport}
Let $T \in \TT$ and let $\xi: \Dom \to \RR$ be continuous.  Define $\bar{\xi}$ from $\xi$ by
\beq
\bar{\xi}(x) := \begin{cases} 
    \xi(x_\ell) 
    & 
    \text{ if } x \le x_\ell,
    \\
    \xi(x)      
    & 
    \text{ if } x_\ell < x < x_r,
    \\
    \xi(x_r) & \text{ if } x_r \le x.
    \\               \end{cases}           \label{eq:extrap}
\eeq
Then, we define $T^\flat \xi: \Dom \to \RR$ by $T^\flat \xi (x) := T^\flat \bar{\xi}(x)$.
\end{definition}

That the transport maps $T \in \TT$ are defined from $\RR$ to $\RR$ is 
different from previous works employing spatial transformations
\cite{Cagniart2019,taddei19}.

Let us construct a map $\hat{T}: \RR \to \RR$ by taking a linear combination of
given linearly independent continuous piecewise $C^1$ maps $v_m: \RR \to
\RR$ ($m=1,...,M$),
\beq
\hat{T}(x) := \sum_{m=1}^M \alpha_m v_m(x)\,.
\label{eq:Tpb}
\eeq
In general it does not hold that $\hat{T} \in \TT$ even if $\{v_m\}_{m=1}^M
\subset \TT$, since $\TT$ is not a vector space. We will focus on $\hat{T}
\in \Span\{v_m\}_{m=1}^M$ that also satisfies $\hat{T} \in \TT$, so it will
be convenient to use the shorthand $\Span_{\TT} \{v_m\}_{m=1}^M := \Span
\{v_m\}_{m=1}^M \cap \TT$.

\begin{definition}\label{def:TTrb}
We call a non-empty set $\bbS$ a \emph{finite-dimensional subset of
transport maps in $\TT$}, with the notation $\bbS \sqsubset \TT$, if
\begin{enumerate}[label=(\roman*)]
    \item $\bbS$ is set of a single member of $\TT$, when we say $\bbS$
has dimension $0$ ($\dim(\bbS) = 0$),
    \item $\bbS = \Span_\TT \{v_m\}_{m=1}^M$ with linearly independent
continuous piecewise linear maps $v_m: \RR \to \RR$ ($m=1, ..., M$), in which
case we say $\bbS$ has dimension $M$ ($\dim(\bbS) := M$) and refer to
$\{v_m\}_{m=1}^M$ as a \emph{basis} of $\bbS$.
\end{enumerate}
\end{definition}

Note that by construction, we require that $\hat{T} \in \bbS \sqsubset \TT$ be
increasing, i.e. $x_1 < x_2$ implies $\hat{T}(x_1) < \hat{T}(x_2)$. But more
generally, $\hT \in \Span \{v_m\}_{m=1}^M$ with arbitrary coefficients
$\{\alpha_m\}_{m=1}^M$ in \cref{eq:Tpb} is not necessarily increasing. The
consequence is that the inverse of $\hT$ becomes undefined. A sufficient
condition to satisfy $\hT \in \TT$ is $\hT' > 0$, which results in a set of
constraints on its coefficients $\{\alpha_m\}_{m=1}^M$. The definition of
$\dim(\bbS)$ agrees with the Lebesgue covering dimension of $\bbS$
\cite{HW41}.

\subsection{The Kolmogorov $(N,M)$-width}

We consider the nonlinear generalization of the reduced approximation \cref{eq:rb_repr} in which the basis functions $\{\zeta_n\}_{n=1}^N$ of $\VVrb$ are composed with the inverse of $T \in \TT$. That is, we consider approximating elements of the solution manifold $\cMh$ by functions lying in the set 
\begin{equation}
    \cV(\VVrb,\TT) 
    := 
    \bigcup_{T \in \TT} \VVrb(T),
    \quad
    \VVrb(T) 
    :=
    \{ T^\flat \xi :  \xi \in \VVrb \}.
    \label{eq:curlyV}
\end{equation}
In particular, we can restrict $\cV(\VVrb,\cdot)$ to use transport maps in $\TTrb \sqsubset \TT$ (\cref{def:TTrb}), resulting in the set $\cV(\VVrb,\TTrb)$. This leads to a generalization of the notion of the Kolmogorov $N$-width \cref{eq:nwidth}.
\begin{definition} \label{def:nmwidth}
    Define the \emph{Kolmogorov $(N,M)$-width} as
    \beq
    d(N,M; \cMh) =
                \inf_{\substack{\VVrb \subset \VV \\ \dim(\VVrb) = N \\ 
                \TTrb \sqsubset \TT \\ \dim(\TTrb) = M}} 
                \sup_{\uh \in \cMh} 
                \inf_{v\in \cV(\VVrb,\TTrb)}
                \Norm{\uh -  v}{\VV}.
    \label{eq:nmwidth}
    \eeq
\end{definition}
It is immediate that
\begin{equation}
d(N;\cM_\delta) = d(N,0;\cM_\delta), 
\quad
d(N;\cM_\delta) \ge d(N,M;\cM_\delta) \text{ for all } M \in \NN.
\end{equation}
Intuitively speaking, the low-rank basis $\{\zeta_n\}_{n=1}^N$ of $\VVrb$ spans the \emph{local} linearization of the solution manifold, whereas the \emph{global} transport-dominated behavior of the solution is approximated by the basis $\{v_m\}_{m=1}^M$ of $\TTrb$. This overall strategy is illustrated as a diagram in \cref{fig:diagram-param}. 

There are various nonlinear widths proposed in the past. In comparison,
the nonlinear approximation class with which the $(N,M)$-width is measured is
significantly more specialized, and yields bounds that are more specific to the
solution manifolds of transport-dominated PDEs like \cref{eq:prob}. For
instance, whereas the $(N,M)$-width is the smallest achievable approximation
error using $\cV(\VVrb,\TTrb)$, a more general class is used in the definition
of the nonlinear $n$-width in \cite{DeVore89, DeVore93, Cohen20}. Arbitrary
continuous mappings of the solution manifold to and from $\RR^n$ is considered.
An approximation in $\cV(\VVrb,\TTrb)$ can be identified to a point in
$\RR^{M+N}$, and with additional continuity restriction, any lower bound for
the nonlinear $n$-width, e.g. for the unit ball in Besov spaces
\cite{DeVore89}, would imply a bound for the restricted version of the
$(N,M)$-width. Conversely, the constructive approximations in
$\cV(\VVrb,\TTrb)$ imply an upper bound for the nonlinear $n$-width.

The $(N,M)$-width is also related to the nonlinear Kolmogorov width
of a similar name introduced in \cite{Temlyakov98}, where multiple
subspaces with identical dimension are used in the approximation of $\cMh$.
However, the nonlinear Kolmogorov width of $\cV(\VVrb,\TTrb)$ can be large in
general, even when the dimensions of $\VVrb$ and $\TTrb$ are both small (e.g.
consider $\VVrb = \Span \{\zeta_1\}$ with compactly supported $\zeta_1 \in
\VV$ and $\TTrb = \Span_\TT \{1, \Id\}$).

\begin{figure}
\centering
\includegraphics[width=0.5\textwidth]{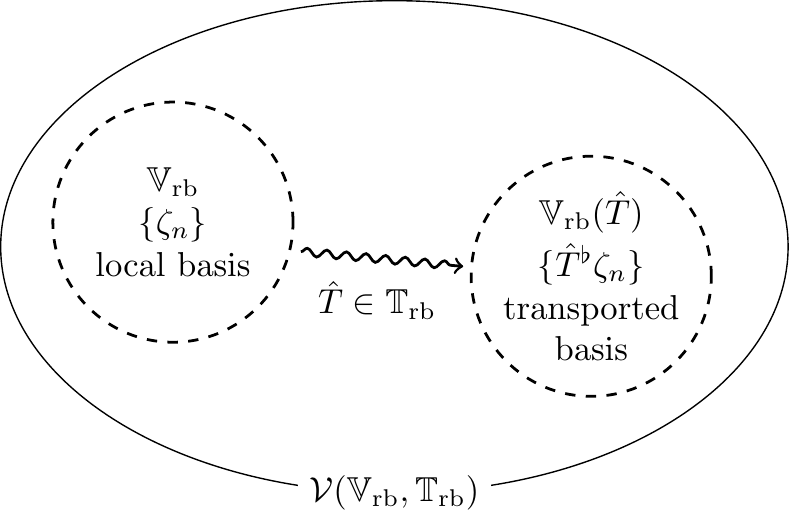}
\caption{Diagram depicting the local and global approximations using  $\cV(\VVrb,\TTrb)$, which is generated by transporting the subspace $\VVrb$ using $\hT \in \TTrb$.}
\label{fig:diagram-param}
\end{figure}

For systems of conservation laws, the scalar problem \cref{eq:prob} is extended into multiple state variables: the solution $u$ takes on vector values and the flux function $f(\cdot,x;\bmu)$ satisfies a generalized convexity condition \cite{fvmbook}. In this case, each state variable is a superposition of time-dependent characteristic variables, each with a corresponding set of characteristic curves. As a result, the approximations in $\cV(\VVrb,\TTrb)$, as well as the corresponding notion of $(N,M)$-width, is no longer sufficient. But an extension is possible by a superposition.

\begin{definition} 
    \emph{The Kolmogorov $(N,M,L)$-width} is given by
    \beq
    d(N,M,L; \cMh) =
      \inf_{\substack{
      \VVrb \subset \VV \\ \dim(\VVrb) = N \\ 
      \TTrb \sqsubset \TT \\ \dim(\TTrb) = M}} 
      \sup_{\uh \in \cMh} 
      \inf_{v_\ell \in \cV(\VVrb,\TTrb)}
      \Norm{\uh - \sum_{\ell=1}^L v_\ell}{\VV}.
    \label{eq:nmlwidth}
    \eeq
\end{definition} 
    
Again, $d(N,M,1;\cM_\delta) = d(N,M;\cM_\delta)$, $d(N,M;\cM_\delta) \ge
d(N,M,L;\cM_\delta)$. Various extensions to multiple spatial dimensions are
possible, for example by letting the transport map to be a diffeomorphism in
higher dimensions \cite{iollo14,taddei19,nonino19}. For hyperbolic problems in
particular, one can use the Radon transform \cite{rim18a,rim18c}.

\subsection{Decay of Kolmogorov $(N,M)$-width}

For intuition, we briefly discuss the decay of the $(N,M)$-width with examples.

\begin{example}[Linear advection] \label{expl:advection}
Consider a 1-parameter linear advection problem $f(u;\mu_1) = \mu_1 u, \psi =
0$ and $u_0(x) = \phi(x)$, for which $(\bfeta) = (t,\mu_1)$. The solution
manifold has a small $(N,M)$-width, which can be seen by letting $\VVrb :=
\Span\{\phi(x)\}$ and $\TTrb = \Span_\TT \{1,\Id(x)\}$. Then $d(1,2;\cM_\delta) = \mathcal{O}(\delta^r)$, achieving an error at the level of the full-model discretization error \cref{eq:fm_error}.
\end{example}

\begin{example}[Gibbs phenomenon] 
If we further restrict $\VVrb$ in \cref{eq:nwidth} and \cref{eq:nmwidth}
to be a subset of the Chebyshev basis of dimension $N$ and define the corresponding $N$- and $(N,M)$-width as $\tilde{d}(N;\cMh)$ and $\tilde{d}(N,M;\cMh)$, it is known that for a solution manifold containing the signum function $\sign: [-1,1] \to \{-1,0,1\}$ with a jump at $x = 0$, it holds that $\tilde{d}(N;\cM_\delta) \ge c/N$ for some $c > 0$; see \cite{atap-book}. However, allowing the diffeomorphism $\TTrb = \{x^{2q+1}\}$ with large enough $q \in \NN$ to transform $\VVrb = \{x\}$ has the consequence that $\tilde{d}(1,1;\cM_\delta) = \cO(\delta)$.
\end{example}

\begin{example}[Burgers' turbulence] \label{expl:burgers_random}
One can show that a continuous piecewise linear discretization of white noise,
in the form of \cref{eq:fom_repr} with $\{\varphi_n\}$ nodal basis functions
and independent coefficients $\{b_n\}$ drawn from normal distributions,
does not have a small $(N,M)$-width for small $N$ and $M$. Considering the
solution of the Burgers' equation with such a random initial condition, it can
be shown that the solution manifold for this problem has a large $(N,M)$-width
during a time interval $[0,\eps)$ for which $\eps$ is small enough.  
\end{example}

\begin{example}[Wave equation] \label{expl:wave}
The wave equation $\partial_{tt} u + \partial_{xx} u = 0$ can be written in
a first order system of the type \cref{eq:prob} with two state variables (see,
e.g.  \cite{fvmbook}), and thus can propagate waves in two opposite
directions. As discussed above, in this case we expect that
$d(N,2;\cM_\delta)$ will decay slowly with respect to $N$ since one
transport map cannot well-represent in a low-rank manner a superposition of
two waves traveling at different speeds in general. However, for zero
initial velocity and radiation boundary conditions, one can show
$d(1,2,2;\cM_\delta) = \cO(\delta^r)$. 
\end{example}

Additional details regarding \cref{expl:burgers_random,expl:wave} can be
found in \cref{sec:details}.

\subsection{Reduced nonlinear approximation}

A fast decay in the $(N,M)$-width is only a non-constructive statement. In later sections, we propose a constructive method that finds an approximation in $\cV(\VVrb,\TTrb)$. Given  $\uh(\cdot,t;\bmu) \in \cMh$, the method provides an algorithm for computing the reduced approximation $\hat{u}(\cdot,t;\bmu) \in \cV(\VVrb,\TTrb)$. Then, for each fixed $(t,\bmu) \in \cH$, $\hat{u}$ can be expressed as
\begin{equation}
    \left\{
    \begin{aligned}
    \hat{u}(x,t;\bmu) &= 
    \sum_{n=1}^N \beta_n(t,\bmu) \hTpb^\flat \zeta_n(x),
    \\
    \hTpb(x) &= \sum_{m=1}^M \alpha_m(t,\bmu) v_m(x).
    \end{aligned}
    \right.
    \label{eq:repr-Tpb}
\end{equation}
The reduced approximation $\hat{u}(\cdot,t;\bmu)$ is described by two sets of coefficients: the transport coefficients $\{\alpha_m(\bfeta)\}_{m=1}^M$ and the local basis coefficients $\{\beta_n(\bfeta)\}_{n=1}^N$. The reduced approximation  \cref{eq:repr-Tpb} relies on a different reduced subspace $\VVrb(\hTpb)$ for different $(\bfeta) \in \cH$. We call this subspace corresponding to $(\bfeta)$ a \emph{transported subspace} as described in the following definition.
\begin{definition} \label{def:tsubsp}
Let us denote by $\hTpb \in \TTrb \sqsubset \TT$ a transport map corresponding to $(\bfeta) \in \cH$. The \emph{transported subspace} is a space generated by compositions of functions in the space $\VVrb$ with the inverse of $\hTpb$,
\begin{equation}
    \VVrb(\bfeta)
    :=
    \VVrb(\hTpb)
    = 
    \{\hTpb^\flat \xi: \xi \in \VVrb\},
    \quad
    (\bfeta) \in \cH.
\end{equation}
\end{definition}
    
Note that $\VVrb(\bfeta) \not \subset \VVh$ in general, so the transported subspaces are not necessarily conforming with respect to the full-model space $\VVh$.

\section{Offline construction of low-rank subset of transport maps} \label{sec:offline}

In this section, we discuss a procedure that constructs the bases
$\{v_m\}_{m=1}^M$ of $\TTrb$ and $\{\zeta_n\}_{n=1}^N$ of $\VVrb$ to be used in
the nonlinear reduced approximation \cref{eq:repr-Tpb}. A key feature of our
procedure is that it relies only on well-understood linear dimensionality
reduction techniques such as the singular value decomposition (SVD), rather
than requiring the solution of special, potentially non-convex, optimization
problems. That is, while it is not the only possible procedure to compute
the transport modes, it is direct and well-posed whenever applicable.

We first recall definitions related to a nonlinear interpolation procedure called displacement interpolation by pieces (DIP) introduced in \cite{rim18b,rim19b}. The interpolation proceeds by decomposing $\uh(\cdot, t;\bmu)$ into a sum of monotone functions, then computing the optimal transport map between the corresponding monotone functions. One combines these maps to obtain the DIP maps, a set of transport maps in $\TT$. By constructing a low-rank approximation to the DIP maps, we obtain the basis $\{v_m\}_{m=1}^M$ to be used in the nonlinear reduced approximation \cref{eq:repr-Tpb}. We then discuss the construction of the basis $\{\zeta_n\}_{n=1}^N$.

\begin{remark} \label{rmk:full}
From here and onwards, we will assume that the full model is first order accurate ($r=1$ in \cref{eq:fm_error}), since a finite volume method (FVM) with Godunov flux is total variation diminishing (TVD) without requiring additional modifications, e.g. limiters or non-oscillatory higher-order reconstructions. As a result, the full model as well as the corresponding estimate \cref{eq:fm_error} extends to the non-classical solutions of \cref{eq:prob} without changes (see \cite{fvmbook}). Therefore, this choice simplifies our exposition without limiting the extensibility to the non-classical case. We take as $\uh$ the continuous piecewise-linear reconstruction of such a FVM solution in the form \cref{eq:fom_repr}. The constructive methods proposed in this and subsequent sections can be extended to the higher-order case ($r > 1$) with more work. 
\end{remark}

\begin{figure}
\centering
\begin{tabular}{cc}
\includegraphics[width=0.8\textwidth]{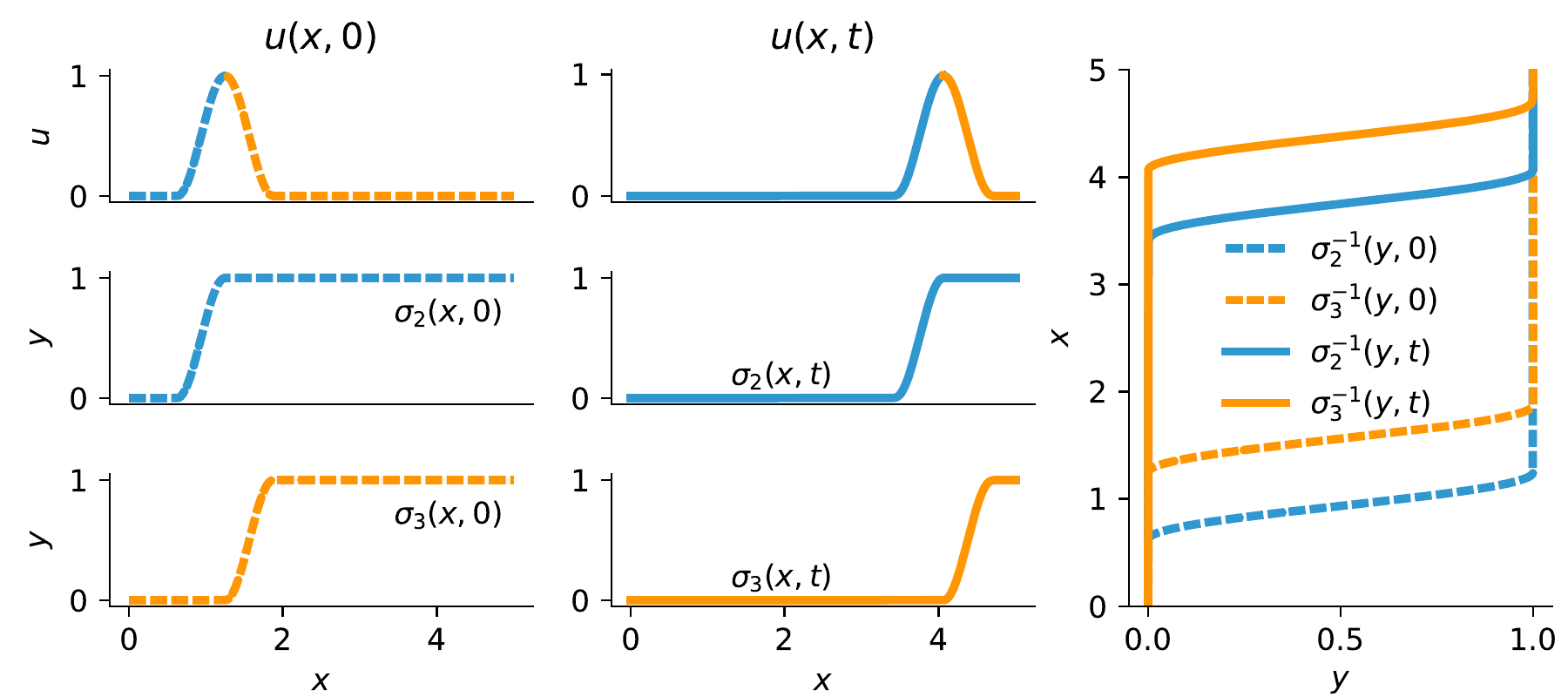}
\end{tabular}
\caption{
The monotone decomposition \cref{eq:decomp} for the solution $u(x,t)$ to the
advection equation for two different times (left) with $\gamma_2 =1, \gamma_3 =
-1$. The inverse of $\sigma_2,\sigma_3$ (right). For additional details, see
\cref{appdx:expl_advec}.
}
\label{fig:advection_sig}
\end{figure}

\subsection{Monotone decomposition} \label{sec:decomp}
In order to define DIP, we first define the \emph{monotone decomposition} which decomposes the solution into a sum of monotone functions. We start by making a preliminary statement that the continuous piecewise-linear $\uh(\cdot,t;\bmu)$ can be decomposed as a sum of certain non-decreasing functions. 

\begin{lemma} 
\label{lem:decomp}
One can represent $\uh(\cdot,t;\bmu) \in \cMh$ \cref{eq:fom_repr} as a sum
\begin{equation}
  \uh(x,t;\bmu) = \uh(x_*,t;\bmu) + \sum_{j=1}^{J} \gamma_j(t,\bmu) \sigma_j(x,t;\bmu),
    \label{eq:decomp}
\end{equation}
where $x_* \in \Dom$ is a fixed point, $\gamma_j: (0,\tfin) \times \cD \to
\RR$, and for each $(t,\bmu) \in \cD$ the functions $\sigma_j =
\sigma_j(\cdot,t;\bmu): \Dom \to [0,1]$ satisfy:
\begin{enumerate}[label=(\roman*)]
\item $\sigma_j$ is continuous, piecewise linear, and non-decreasing,
\item $\sigma_j(x_\ell) = 0$ and $\sigma_j(x_r) = 1$,
\item strictly increasing in the preimage $\sigma_j^{-1}((0,1))$,  
\item $\setint(\supp \sigma_j') \cap \setint(\supp \sigma'_k) = \emptyset$ if $j \ne k$ where $\setint(\cdot)$ is the interior of a set, 
\item $j < k$ implies $\sup \supp \sigma_j' \le \inf \supp \sigma_k'$,
\item $\Dom = \cup_{j=1}^J \supp \sigma_j'$.
\end{enumerate}
\end{lemma}


\begin{proof}
Recall the expression \cref{eq:fom_repr} for $\uh(x,t;\bmu)$, where we let $\varphi_n(x)$ be the hat function, i.e. continuous piecewise linear function such that $\varphi_n(x_\ell + (m-1)\Delta x) = \delta_{n,m}$ where $\delta_{n,m}$ is the Kronecker delta and $n,m = 1, ...,N_\delta$. Now, let $\varsigma_n := \sum_{m = n}^{N_\delta} \varphi_{m}$ then
\beq
\uh(x,t;\bmu) 
= 
\sum_{n=1}^{N_\delta} \beta_n(t,\bmu) \varphi_n(x)
= \beta_1(t,\bmu) \varsigma_1(x)
+
\sum_{n=2}^{N_\delta} (\beta_n(t,\bmu) - \beta_{n-1}(t,\bmu)) \varsigma_n(x),
\eeq
so that $x_* := x_\ell$, $J := N_\delta -1$, $\sigma_j := \varsigma_{j+1}$,  $\gamma_j := \beta_{j+1} - \beta_j$ gives,
\beq
u_\delta(x,t;\bmu)
=
u_\delta(x_*,t;\bmu)
+ 
\sum_{j=1}^J \gamma_j(t,\bmu) \sigma_j(x).
\eeq
Note that $\sigma_j$ is continuous, piecewise linear, non-decreasing function,  $\sigma_j(x_\ell) = 0$, $\sigma_j(x_r) = 1$, and $\supp \sigma'_j = [x_\ell + (j-1) \Delta x, x_\ell + j \Delta x]$ with $\sigma_j^{-1}((0,1)) = \setint(\supp \sigma'_j)$ in which it is strictly increasing with $\sigma'_j = 1/\Delta x$.
\end{proof}

\begin{definition} \label{def:decomp}
A \emph{monotone decomposition} of $\uh(\cdot,t;\bmu) \in \cMh$ is the sum \cref{eq:decomp} with the minimal $J \in \NN$.
\end{definition}

For an illustration of the monotone decomposition of solutions to the advection equation, see \cref{fig:advection_sig}. Since $u_\delta$ is continuous and piecewise linear in our setting, an exact decomposition \cref{eq:decomp} is available. In more general settings, the identity in \cref{eq:decomp} can be relaxed to an approximation. 

Note that one can extrapolate $\sigma_j(\cdot,t;\bmu)$ defined in the bounded domain $\Dom = (x_\ell, x_r)$ to the real line $\RR$, by letting  $\sigma_j(x,t;\bmu)$ be $0$ if $x \le x_\ell$ and $1$ if $x \ge x_r$. This extends $\uh(\cdot;\bfeta)$ also to the real line (equivalent to \cref{eq:extrap}). Also note that since the full model is assumed to be TVD, the corresponding $u_\delta$ for the homogeneous version of \cref{eq:prob} has $J$ decreasing over time.

\subsection{Signature condition and DIP maps} \label{sec:sgn}

We introduce the signature condition which is relevant when comparing the decomposition \cref{eq:decomp} for $\uh(\cdot,t;\bmu) \in \cMh$ across different values of $(\bfeta) \in \cH$. Let us collect snapshots of $u_\delta(\cdot,t;\bmu)$ at various values of $(\bfeta) \in \cHg$ designated by the set
\beq
\Htrain := \{(\bfetai{1}), (\bfetai{2}), ..., (\bfetai{\Ntrain})\}.
\label{eq:cHg}
\eeq
Then, we denote the set of snapshots of $\uh$ by
\beq
    \cSg := \{ \uh(\cdot,t;\bmu) : (\bfeta) \in \cHg \}.
    \label{eq:T_snapshot}
\eeq
Now, we define the signature condition on $\cSg$.

\begin{definition}[Signature condition]
\label{def:sgn}
A set of snapshots $\cSg$ \cref{eq:T_snapshot} satisfies the
\emph{signature condition} if each snapshot in the set has a monotone
decomposition (\cref{def:decomp}) that holds with the same $x_* \in \Dom$ and
and the same vector $(\sign(\gamma_j))_{j=1}^J$.
\end{definition}

If $\cSg$ satisfies the signature condition, each snapshot $\uh \in \cSg$ can be expanded by the decomposition, 
\beq
\uh(x,t_\ell;\bmu_\ell)
= 
\uh(x_*,t_\ell;\bmu_\ell) 
+ 
\sum_{j=1}^J \gamma_j(t_\ell,\bmu_\ell) \sigma_j(x,t_\ell;\bmu_\ell),
\quad
\ell = 1, ..., \Ng,
\label{eq:sgn_decomp}
\eeq
for $x_*$ and $J$ independent of $\ell$ and $\sign( \gamma_j(t_\ell,\bmu_\ell))
= \sign(\gamma_j(t_{\ell'},\bmu_{\ell'}))$ for $\ell, \ell' = 1, \dots, \Ng$.
Then, for a fixed $j \in \{1, ..., J\}$, every pair
$\sigma_j(\cdot,t_\ell;\bmu_\ell),\sigma_j(\cdot,t_{\ell'};\bmu_{\ell'})$
yields a specific transport map that is the explicit solution to the
Monge-Kantorovich optimal transport problem, called the \emph{monotone
rearragement} \cite{villani03,villani2008}. 

We will combine these maps for $j=1,...,J$ to produce DIP maps. Let us
denote $\sigma_j(x) := \sigma_j(x,t_1;\bmu_1)$. Then the rearrangement map
from $(\bfetai{1})$ to $(\bfetai{\ell})$ is the map $R_{j,\ell}(x): \supp
\sigma_j' \to \RR$ given by
\beq
R_{j,\ell}(x) := \sigma_j^{-1}(\sigma_j(x),t_\ell;\bmu_\ell),
    \label{eq:Rjl}
\eeq
where $j$ is the index in the monotone decomposition \cref{eq:sgn_decomp}. Note that $R_{j,\ell}$ is continuous and piecewise linear: since $\sigma_j^{-1}(\cdot,t;\bmu)$ is continuous and piecewise linear, $R_{j,\ell}$ is a composition of two such maps. So far, the map is only defined in a part of the domain, so we combine these maps for all $j=1, ..., J$. Let $\Dom_0 := \cup_{j=1}^J \setint(\supp \sigma'_j)$, then let $R_\ell(x): \Dom_0  \to \RR$ be defined by
\beq
    R_\ell(x) := R_{\mathfrak{j}(x),\ell}(x) 
    \label{eq:Rl}
\eeq
where $\mathfrak{j}(x)$ is $\mathfrak{j} \in \{1, ..., J\}$ such that
$x \in \setint(\supp \sigma'_j)$.  The DIP map is the extension of $R_{\ell}$
on the entire real line as defined in the following.

\begin{definition} \label{def:dip_Tl}
The \emph{DIP map} $T_\ell$, $\ell = 1, .., \Ng,$ is the continuous piecewise
linear extension of $R_\ell$ given in \cref{eq:Rl} to the real line $\RR$ by
letting $R_\ell$ be linear outside $\Dom$ and
\begin{enumerate}[label=\emph{(\arabic*)}]
\item continuation at $\setint(\Dom) \setminus \Dom_0$,
\item extrapolation at the boundaries by $T_\ell'(x) :=  R_\ell'(x)$ for $x \in \{x_\ell,x_r\}$.
\end{enumerate}
\end{definition}

As discussed in \cite{rim19b}, it can be shown that the DIP map approximates the map generated by the characteristic curves of the scalar conservation law \cref{eq:prob}. A consequence is that DIP between the classical solutions to the homogeneous version of \cref{eq:prob} at two different times yields the solution at an intermediate time.

\subsection{Low-rank subset of transport maps}\label{sec:lr-tmodes}

The DIP maps $\{T_\ell\}_{\ell=1}^\Ng$ can be low-rank, although the set of functions generated by the composition of the DIP map $\{T_\ell^\flat \xi: \xi \in \VVrb\}_{\ell=1}^\Ng$ is high-rank in general (see \cref{fig:advection_sig}). To exploit this low-rank structure, we proceed to compute the low-rank approximation to $\{T_\ell\}_{\ell=1}^\Ng$.

Collecting the DIP maps $\{T_\ell\}_{\ell=1}^\Ng$
one can use standard tools, e.g. SVD, to compute a low-rank structure. Let $\bfT:\RR \to \RR^\Ng$ with $\bfT := [T_1, T_2, ..., T_\Ng]$ then 
\beq
    v_m(x) = \bfT(x) \cdot \bv_m,
    \label{eq:tmodes}
\eeq
where $\cdot$ is the inner product between two vectors in $\RR^\Ng$, and $\bv_m \in \RR^\Ng$, is the $m$-th eigenvector for the matrix $\bfC \in \RR^{\Ng \times \Ng}$ given by the inner products
\beq
   (\bfC)_{\ell,\ell'} :=  (T_\ell,T_{\ell'}), 
    \quad \ell,\ell' = 1, ..., \Ng.
    \label{eq:Cll}
\eeq
One obtains $\{v_m\}_{m=1}^M$ if one truncates the eigenfunction expansion after the first $M$ terms corresponding to the largest eigenvalues. 
    
Note that $T_1 = \Id$ where $\Id$ denotes the identity, and that each map
$T_\ell$ contains a component of the identity map, i.e. $(T_\ell,\Id) \ne 0$.
So it is convenient to let $v_1 := \Id$ then compute the correlation matrix
of the perturbation from the identity. That is, pre-process $\bfC$ to
obtain $\bar{\bfC} \in \RR^{(\Ng-1) \times (\Ng-1)}$ 
\beq
(\bar{\bfC})_{\ell,\ell'} 
:= 
(T_{\ell+1} - \Id, T_{\ell'+1} - \Id) 
\quad \text{ for } 
\ell,\ell' = 1, ..., \Ng-1,
\eeq
then use $\bar{\bfC}$ in place of $\bfC$ to compute the corresponding basis $\{v_m\}_{m=2}^M$. We shall do so in our numerical experiments below. 

In what follows, we will denote by $\{v_m\}_{m=1}^M$ and $\TTrb \sqsubset \TT$ the basis and finite-dimensional subset constructed in this section.

\begin{definition}\label{def:lr-tmaps}
We will refer to the constructed basis $\{v_m(x)\}_{m=1}^M$ \cref{eq:tmodes} as \emph{transport modes}. We will denote by $\TTrb$ the resulting finite-dimensional subset of transport maps (\cref{def:TTrb}), and we refer to it as the \emph{low-rank subset of transport maps}.
\end{definition}

A detailed example of the DIP maps and the transport modes for the
advection equation (\cref{fig:advection_sig}) appears in
\cref{appdx:expl_advec}.

\subsection{Local basis}

We will briefly discuss how to construct the subspace $\VVrb$. Revisiting the diagram in \cref{fig:diagram-param}, we wish to obtain $\VVrb$ that approximates the manifold $\cMh$ locally. To this end, the set of snapshots we denote by $\cSl$ will be taken from the local manifold. We use the time-parameters $(\bfeta)$ near a small neighborhood of the point $(\bfetai{1}) \in \cHg$ \cref{eq:cHg} to obtain the snapshots, denoted by $\cHl$ 
\beq
    \cHl := \{(\bfetai{1}), ..., (\bfetai{\Nl})\}.
\eeq
For convenience, we construct $\VVrb$ by taking as snapshots
\beq
\begin{aligned}
u_{\delta, \ell} &:=
\uh(\cdot,t_\ell;\bmu_\ell),
\\
\partial_x u_{\delta, \ell} &:=
\partial_x \uh(\cdot,t_\ell;\bmu_\ell), 
\end{aligned}
\quad
\begin{aligned}
\partial_x f_\ell
&:=
\partial_x [f(\uh(\cdot,t_\ell;\bmu_\ell),\cdot\,;\bmu_\ell)],
\\
\psi_\ell
&:=
\psi(\uh(\cdot,t_\ell;\bmu_\ell),\cdot\,;\bmu_\ell),
\end{aligned}
\eeq
for each $(t_\ell,\bmu_\ell) \in \cHl$, then let
\beq
    \cSl := \left\{
        u_{\delta, \ell},
        \partial_x u_{\delta, \ell},
        \partial_x f_\ell,
        \psi_\ell
        : 
        \ell = 1, ... , \Nl 
    \right\}.
    \label{eq:sns_loc}
\eeq
The reduced space $\VVrb$ is spanned by the first $N$ left-singular vectors $\{\zeta_n\}_{n=1}^N$ of the snapshot matrix corresponding to $\cSl$. 

\section{Online reduced model}\label{sec:online}

In this section, we introduce methods that allow efficient online computations with the reduced nonlinear approximation \cref{eq:repr-Tpb}. Compared to the linear case \cref{eq:rb_repr}, additional considerations are necessary due to the nonlinear aspects of the approximation \cref{eq:repr-Tpb}.

First, we introduce moving interpolation points we call \emph{interpolation
particles}. The motivation is to take advantage of the form of the
approximation \cref{eq:repr-Tpb}, by making the interpolation points move
depending on $(t,\bmu) \in \cH$ via the transport maps $\hTpb \in \TTrb$
(\cref{def:lr-tmaps}). We will show that they enable an efficient change of
basis yielding a projection onto the transported subspaces $\VVrb(t,\bmu)$
(\cref{def:tsubsp}). The idea is applied in particular to the empirical
interpolation method \cite{barrault04,chanturantabut10}.

Then we introduce an online time-stepping for reduced models for the problem \cref{eq:prob}. It alternates between the update of the local basis coefficients $\{\beta_n(\bfeta)\}_{n=1}^N$ and the update of the transport mode coefficients $\{\alpha_m(\bfeta)\}_{m=1}^M$. The cost of time-stepping the nonlinear reduced model depends only on the reduced degree of freedoms $N$ and $M$, and does not depend on the size of the full model $N_\delta$ at each time-step.

\subsection{Interpolation particles} \label{sec:ipts}

To deal with nonlinear functions in our problem \cref{eq:prob} as well as function compositions in \cref{eq:repr-Tpb} in an efficient manner, we introduce a moving interpolation method that is compatible with empirical interpolation. In the traditional empirical interpolation, one uses the basis functions $\{\zeta_n\}_{n=1}^N$ of $\VVrb$ and the corresponding interpolation points $\cX := \{x_i\}_{i=1}^N \subset \Dom$.  We generalize the pair $\{\zeta_n\}_{n=1}^N$ and $\{x_i\}_{i=1}^N$ by allowing the interpolation to move according to the transport map $\hTpb \in \TTrb$. As a result, one obtains a set of pairs: the transported basis $\{\hTpb^\flat \zeta_n \}_{n=1}^N$ and the corresponding interpolation points $\hTpb(\cX)$. 

Hence, the transport evolves the interpolation points over time, endowing these points with momenta. So we will refer to these interpolation points $\hTpb(\cX)$ that move along the transport as \emph{interpolation particles}. 

\begin{definition}
Given a basis $\{\zeta_n\}_{n=1}^N$ that spans $\VVrb$ and its interpolation points $\cX =\{x_i\}_{i=1}^N \subset \Dom$, the \emph{interpolation particles} for the basis $\{\hTpb^\flat \zeta_n\}_{n=1}^N$ of $\VVrb(\bfeta)$ is defined as $\hTpb(\cX)$.
\end{definition}

Onwards, we will refer to the interpolation particles simply as particles. A useful property is that the empirical interpolation construction commutes with the transport operation. 

\begin{proposition}
Given a basis $\{\zeta_n\}_{n=1}^N$ and corresponding interpolation points $\cX =\{x_i\}_{i=1}^N \subset \Dom$, assume $\hTpb \in \TTrb \sqsubset \TT$ and that $\hTpb(\cX) \subset \Dom$. Then the particles $\hTpb(\cX)$ are the interpolation points of $\{\hTpb^\flat \zeta_n\}_{n=1}^N$.
\end{proposition}

\begin{proof}
For simplicity, denote $\hat{T} := \hTpb$ for a fixed $(t,\bmu) \in \cH$ during this proof. For $x_i \in \cX$, 
\beq
\hT^\flat \zeta_n( \hT(x_i) ) = \zeta_n(\hT^{-1}(\hT(x_i))) =  \zeta_n (x_i)
\quad
\text{ for } i = 1,2, ..., N,
\eeq
which implies due to the construction of $\cX$ that
\beq
    \argmax_{x \in \Dom} \abs{\hT^\flat \zeta_n (x)} 
    =  
    \hT(\argmax_{x \in \hT(\Dom) \cap \Omega} \abs{\zeta_n (x)}).
\eeq
Then we show by induction that the interpolation points of the basis
$\{\hat{T}^\flat \zeta_n\}_{n=1}^N$ is equal to $\hat{T}(\cX)$. Let us denote by $\{\tau_1, ..., \tau_I\}$ the interpolation points corresponding to $\{\hT^\flat \zeta_1, ..., \hT^\flat \zeta_I\}$ for $I \le N$. When $I=1$,
\beq
\tau_1 = \argmax_{x \in \Dom} | \hT^\flat \zeta_1(x)|
= \hT(\argmax_{x \in \hT(\Dom) \cap \Omega} |\zeta_1(x)|)
= \hT(x_1).
\eeq
Suppose $\tau_i = \hat{T}(x_i)$ for $i = 1,2,...,I-1$. Then upon solving for $\{\vartheta_n\}_{n=1}^{I-1}$ in
\beq
    \sum_{n=1}^{I-1} \vartheta_n^{I-1} \hT^\flat \zeta_n(\tau_i)
    =
    \hT^\flat \zeta_I(\tau_i)
    \quad
    \text{ for } i = 1,2, ..., I-1,
    \label{eq:interp1}
\eeq
the next interpolation point $\tau_{I}$ is found by
\beq
\tau_I
=
\argmax_{x\in \Dom}
|\hT^\flat \zeta_I(x) 
- 
\sum_{n=1}^{I-1} \vartheta_n^{I-1} \hT^\flat \zeta_n(x)|.
\label{eq:interp2}
\eeq
Observe that \cref{eq:interp1} and \cref{eq:interp2} can be re-written in terms of the next interpolation point $x_I$ of $\{\zeta_n\}_{n=1}^{I}$. That is, since $\{\vartheta_n\}_{n=1}^{I-1}$ also satisfy,
\beq
    \sum_{n=1}^{I-1} \vartheta_n^{I-1} \zeta_n(x_i)
    =
    \zeta_I(x_i)
    \quad
    \text{ for } i = 1,2, ..., I-1,
\eeq
one may re-write \cref{eq:interp2} and obtain
\beq
\tau_I
=
\hT(
\argmax_{x\in \hT(\Dom)}
|\zeta_I(x) 
- 
\sum_{n=1}^{I-1} \vartheta_n^{I-1} \zeta_n(x)|)
= \hT(x_I).
\eeq
Therefore $\tau_i = \hT(x_i)$ for all $i = 1,2, ... ,N$.
\end{proof}

Moreover, it is efficient to compute the derivatives of $\xi \in \VVrb(t,\bmu)$ at the particle locations $\hTpb(\cX)$.

\begin{proposition} \label{prop:diff}
If $\xi \in \VVrb(t,\bmu)$ then $\xi'(\hTpb(x_i))$ for each $x_i \in \cX$ can be computed in $\cO(N + M)$ operations.
\end{proposition}

\begin{proof}
Since $\xi \in \VVrb(t,\bmu)$, we have $\xi'(x) = \sum_{n=1}^N \vartheta_n (\hTpb^\flat \zeta_n)'(x)$. Furthermore, the derivative  $(\hTpb^\flat \zeta_n)'$ at the particle location $\hTpb(x_i)$ is given by the chain rule,
\beq
     (\hTpb^\flat \zeta_n)' (\hTpb(x_i)) 
    = 
    \frac{1}{\hTpb'(x_i)} \zeta_n'(x_i)
    \quad 
    \text{ for } 
    i,n = 1, ..., N.
\eeq
Since $\hTpb \in \Span\{v_m\}_{m=1}^M$  \cref{eq:repr-Tpb}, its derivative at the particle locations is 
\begin{equation}
    \hTpb'(x_i) 
          = \sum_{m=1}^M \alpha_m(\bfeta) v_m'(x_i)
    \label{eq:dTpb}
    \text{ for } 
    i = 1, ..., N.
\end{equation}
Thus $\xi'(x_i)$ is computed in $3N+ 2M-2$ operations, that is, $\cO(N+M)$.
\end{proof}

\begin{remark}
In computing $\xi'(\hTpb(x_i))$ the values $\{\zeta_n'(x_i)\}_{n=1}^N$, $\{v_m'(x_i)\}_{m=1}^M$ do not depend on $(t,\bmu)\in \cH$, so they can be pre-computed during the offline stage.
\end{remark}

\subsection{Change of basis}
\label{sec:cob}

When the transport map $\hTpb$ and therefore $\VVrb(t,\bmu)$ is updated, the approximation \cref{eq:repr-Tpb} allows a change of basis procedure with computational effort of $\cO(N+M)$ and does not require expensive calculations that depend on the size of the full model $N_\delta$. Note that the following change of basis is approximate and preserves the function values at the particles only.

Here we describe how to perform the change of basis. Given the local basis
$\VVrb \subset \VVh$, interpolation points $\cX = \{x_i\}_{i=1}^N$, transport
maps  $\hat{T}_{(1)},\hat{T}_{(2)} \in \TTrb$ and corresponding particles
$\hat{T}_{(1)}(\cX)=\{x_i^{(1)}\}_{i=1}^N$,
$\hat{T}_{(2)}(\cX)=\{x_i^{(2)}\}_{i=1}^N$, suppose $u^{(1)} \in
\VVrb(\hat{T}_{(1)})$ and we wish to compute $u^{(2)} \in
\VVrb(\hat{T}_{(2)})$, satisfying $u^{(1)}(x_i^{(2)}) = u^{(2)}(x_i^{(2)})$ for
all $i = 1, ..., N$. 

Writing $u^{(1)},u^{(2)}$ and $T^{(1)},T^{(2)}$ as
\begin{equation}
    \begin{aligned}
    u^{(1)}(x) &= \sum_{n=1}^N \beta_n^{(1)} 
                    \hat{T}_{(1)}^\flat \zeta_n(x),
    \quad
    \hat{T}_{(1)}(x) = \sum_{m=1}^M \alpha_m^{(1)} v_m(x),
    \\
    u^{(2)}(x) &= \sum_{n=1}^N \beta_n^{(2)} 
                    \hat{T}_{(2)}^\flat \zeta_n(x),
    \quad
    \hat{T}_{(2)}(x) = \sum_{m=1}^M \alpha_m^{(2)} v_m(x),
    \end{aligned}
    \label{eq:cob_new}
\end{equation}
it is clear that, to determine  $u^{(2)}(x)$ it suffices to compute its coefficients $\{\beta_n^{(2)}\}_{n=1}^N$. Moreover, if we can compute the values $u^{(1)}(x_i^{(2)})$ then by empirical interpolation we can solve for $\{\beta_n^{(2)}\}_{n=1}^N$ in
\begin{equation}
    \sum_{n=1}^N  \beta_{n}^{(2)} \hat{T}_{(1)}^\flat \zeta_{n}(x_i^{(2)})
    =
    \sum_{n=1}^N \beta_{n}^{(2)} \zeta_{n}(x_i)
    = u^{(1)}(x_i^{(2)}),
    \quad
    \text{ for } i = 1, ..., N,
\end{equation}
and therefore the desired approximation $u^{(2)}(x)$. So it remains to evaluate $u^{(1)}$ at $x_i^{(2)}$, or equivalently, $\hT_{(1)}^\flat \zeta_n$ at $x_i^{(2)}$.

Note that $u^{(1)}$ is piecewise linear on a non-uniform grid, thus evaluating it at an arbitrary point $x_i^{(2)} \in \Dom$ entails an operation that scales as $\cO(N_\delta)$. But if $x_i^{(1)}, x_i^{(2)}$ are close enough, the evaluation only involves local calculations on the uniform grid.

\begin{proposition}
If $0 < \pm (\hT_{(1)}^{-1}(x_i^{(2)}) - x_i) < \delta$ then
\begin{equation}
    \hat{T}_{(1)}^\flat \zeta_n(x^{(2)}_i)
          =
    \zeta_n(x_i)
      + \zeta'_n(x_i^\pm)
         (\hT_{(1)}^{-1}(x_i^{(2)}) - x_i), 
    \label{eq:eval}
\end{equation}
in which $x_i^-,x_i^+$ refers to left and right limits.
\end{proposition}

\begin{proof}
This follows directly from the fact that $\zeta_n \in \VVh$ so it is
continuous and piecewise linear on a uniform grid.
\end{proof}

The values $\zeta_n(x_i), \zeta_n'(x_i^\pm)$ are already known and do not
depend on the transport maps. Hence $u^{(1)}(x_i^{(2)})$ can be computed using
$\cO(MN+N^2)$ values near the particles that can be pre-computed during the
offline stage.

Note that if $x^{(2)}_i$ lies further away from $x_i^{(1)}$, one can compute additional corrections to \cref{eq:eval} to obtain the exact value of $\hat{T}_{(1)}^\flat \zeta_n(x^{(2)}_i)$. Moreover, in case a basis $\{\zeta_n\}_{n=1}^N$ of higher-order polynomial degree was given, corrections can be made to \cref{eq:eval}.

\subsection{Online reduced-model time-update}\label{sec:rom}

In this section, we discuss an online efficient evaluation of the reduced model by time-stepping. We build on Godunov flux updates and employ Godunov splitting for the source term, but other schemes are possible following similar steps as below.

During the evaluation, $\bmu$ is fixed in our PDE \cref{eq:prob}, so we will omit the dependence on the parameter on $\bmu$ by writing, e.g. $\hat{u}(x,t) = \hat{u}(x,t;\bmu)$, $f(u,x) = f(u,x;\bmu)$, $\psi(u,x) = \psi(u,x;\bmu)$, $\hat{T}_{t}= \hTpb$, $\VVrb(t)=\VVrb(t,\bmu)$, $\alpha(t) = \alpha(t,\bmu)$, $\beta(t) = \beta(t,\bmu)$, $\cX(t) = \cX(t,\bmu)$.

We write our reduced-model approximation in a time-discrete form, by discretizing the time variable by uniform time-steps $0=:t_0 < t_1 < ... < t_K := t_\textrm{F}$. Our fully discrete reduced-model approximation becomes
\begin{equation}
    \hat{u}_{(k)}(x) = 
    \sum_{n=1}^N \beta_n^{(k)} \hat{T}_{(k)}^\flat \zeta_n(x),
    \quad
    \hat{T}_{(k)}(x) = \sum_{m=1}^M \alpha_m^{(k)} v_m(x),
    \label{eq:td-rm}
\end{equation}
and corresponding particles are denoted by $\cX^{(k)} = \{x_i^{(k)}\}_{i=1}^N$. 
The discretization is initialized with $\alpha_1^{(0)} = 1$, $\alpha_m^{(0)} = 0$ for $m > 1$ and $\beta_n^{(0)} = (u_\delta(x,0), \zeta_n)/ (\zeta_n,\zeta_n)$ where $(\cdot,\cdot)$ is the inner product in $\mathbb{V}$. The continuous approximation $\hat{u}(x,t)$ \cref{eq:repr-Tpb} can be constructed as the continuous piecewise linear interpolant satisfying $\hat{u}(x,t_k) = \hat{u}_{(k)}(x)$.

The reduced-model time-update during the online stage comprises 3 steps:
\begin{enumerate}[label=(S\arabic*)]
    \item \textbf{Evolution of the PDE.} Update the coefficients $\{\beta_n^{(k)}\}_{n=1}^N$ to $\{\bar{\beta}_n^{(k+1)}\}_{n=1}^N$  using the PDE while keeping the transported subspace $\VVrb^{(k)}$ fixed.
    \item \textbf{Update of the transport map.} Update  $\hat{T}_{(k)}$ to $\hat{T}_{(k+1)}$, that is, $\{\alpha_m^{(k)}\}_{m=1}^M$ to $\{\alpha_m^{(k+1)}\}_{m=1}^M$. This forms the new transported subspace $\VVrb^{(k+1)}$.
    \item \textbf{Change of basis.} Compute $\{\beta_n^{(k+1)} \}_{n=1}^N$ from $\{\alpha_m^{(k+1)}\}_{m=1}^M$ and $\{\bar{\beta}_n^{(k+1)}\}_{n=1}^N$ by a change of basis where the evolved solution is projected onto the new transported subspace $\VVrb^{(k+1)}$.
\end{enumerate}
We detail each step in individual subsections below. A concise algorithmic description of the time-update is given in \cref{sec:alg}.
 
\subsubsection{Evolution of the PDE}
\label{sec:evolve}

In the first step (S1), we update the coefficients $\beta_n^{(k)}$ to $\bar{\beta}_n^{(k+1)}$ while keeping the transport map $\hat{T}_{(k)}$ fixed. This is done by evolving the PDE forward using a Godunov splitting (see \cite{fvmbook}): First we perform the flux update using the Godunov flux, then the source update by adding the contribution from the nonlinear source term. In both cases, the update is projected to the transported basis $\VVrb^{(k)}$. The evolution steps follows in 2 stages \cref{eq:s1-1,eq:s1-2}.

During the flux update the intermediate step $u_*(x)$ with coefficients $\{\beta^*_n\}_{n=1}^N$ is obtained by adding the flux contribution,
\begin{equation}
\beta^*_n  = \beta^{(k)}_n + \lambda \, \theta_n^{(k)},
\quad
\text{ for } n = 1, ..., N,
\tag{S1.1}
\label{eq:s1-1}
\end{equation}
where $\theta_n^{(k)}$ arises from the flux term $\partial_x f$ in the PDE \cref{eq:prob}.

Then the source update is given by 
\begin{equation}
\bar{\beta}_n^{(k+1)} = \beta^*_n + \Delta t \, \omega_n^{(k)}
\quad
\text{ for } n = 1, ..., N,
\tag{S1.2}
\label{eq:s1-2}
\end{equation}
where $\omega^{(k)}$ corresponds to the source term $\psi$ in the PDE \cref{eq:prob}.

In the remainder of this section, we will discuss how to compute the contributions $\{\theta_n^{(k)}\}_{n=1}^N$ and $\{\omega_n^{(k)}\}_{n=1}^N$ using empirical interpolation.

The desired flux update is
\begin{equation}
u_{*}(x) = \hat{u}_{(k)}(x) 
- \lambda \Delta \cF(\hat{u}_{(k)}(x),x),
\end{equation}
where $\lambda = \Delta t / \Delta x$ is the ratio between the size of the time-step and the spatial grid-width of the full model ($\Delta x = \delta$), and the numerical flux difference
$\Delta \cF$ is defined as the upwind flux, 
\begin{equation}
\Delta \cF(\hat{u}_{(k)}(x),x)
:= 
f(\hat{u}_{(k)}(x),x) - f(\hat{u}_{(k)}(x - \Delta x),x - \Delta x).
\end{equation}
To approximate the flux difference $- \Delta \cF(\hat{u}_{(k)}(x),x)$ using the basis $\{\hat{T}_{(k)}^\flat \zeta_n\}_{n=1}^N$ by empirical interpolation, we impose that the two agree at the particles $\{x_i^{(k)}\}_{i=1}^N$,
\beq
    \sum_{n=1}^N \theta_{n}^{(k)} \hat{T}_{(k)}^\flat \zeta_{n}(x^{(k)}_i)
    =
   -\Delta \cF(\hat{u}_{(k)}(x^{(k)}_i),x^{(k)}_i)
    \quad
    \text{ for } i = 1, ..., N.
\eeq
That is, one solves for the coefficients $\{\theta_n^{(k)}\}_{n=1}^N$ in the system
\begin{equation}
\sum_{n=1}^N \theta_{n}^{(k)} \zeta_{n}(x_i) 
    = 
- \Delta \cF(\hat{u}_{(k)}(x^{(k)}_i),x^{(k)}_i)
    \quad
\text{ for } i = 1, ..., N.
\end{equation}
The flux difference $\Delta \cF(\hat{u}_{(k)}(x^{(k)}_i),x^{(k)}_i)$ is then approximated by $f^{(k)}_{\ell,i} - f^{(k)}_{r,i}$ in which
\begin{equation}
\begin{aligned}
f^{(k)}_{\ell,i}
&:= 
f\left(\sum_{n=1}^N \beta_{n}^{(k)} 
(\zeta_{n}(x_i) - (\hat{T}_{(k)}^\flat \zeta_n)' (x_i^{(k)}) \Delta x), x_i^{(k)} - \Delta x 
\right),
\\
f^{(k)}_{r,i}
&:= 
f\left(\sum_{n=1}^N \beta_{n}^{(k)} \zeta_n(x_i),x_i^{(k)} \right).
\end{aligned}
\end{equation}
The derivative in the first term can be computed efficiently (\cref{prop:diff}). This concludes the computation of $\{\theta_n^{(k)}\}_{n=1}^N$.

Next, we compute the contribution from the source term $\psi(u_*(x),x)$. To approximate the contribution with the basis $\{\hat{T}^\flat_{(k)} \zeta_n \}_{n=1}^N$ using empirical interpolation, we solve the system
\begin{equation}
\sum_{n=1}^N \omega_{n}^{(k)} \zeta_{n}(x_i) 
 =  \psi(u_*(x^{(k)}_i),x_i^{(k)})
 \quad
 \text{ for } i = 1, ..., N.
\end{equation}
This yields the coefficients $\{\omega_n^{(k)}\}_{n=1}^N$.

Having updated $\{\beta_n^{(k)}\}_{n=1}^N$ to $\{\bar{\beta}_n^{(k+1)}\}_{n=1}^N$ we let
\begin{equation}
\bar{u}_{(k+1)}(x) := 
\sum_{n=1}^N \bar{\beta}_n^{(k+1)} \hat{T}_{(k)}^\flat \zeta_n(x).
\end{equation}
This completes the first step (S1).

\subsubsection{Update of the transport map}
\label{sec:update_transport}

In the second step (S2), we determine the appropriate subsequent transport map $\hat{T}_{(k+1)}$, that is, we update the transport mode coefficients $\{\alpha_m^{(k)}\}_{m=1}^M$ to $\{\alpha_m^{(k+1)}\}_{m=1}^M$.  The update is given by
\begin{equation}
    \alpha_m^{(k+1)} = \alpha_m^{(k)} + \Delta t \eta_m^{(k)}
    \quad
    \text{ for } m = 1,..., M.
    \tag{S2}
    \label{eq:s2}
\end{equation}
The contribution $\{\eta_m^{(k)}\}_{m=1}^M$ is obtained by solving a system of equations for a subset $\cQ = \{x_{i_j}\}_{j=1}^M \subset \cX$,
\begin{equation}
    \sum_{m=1}^M  \eta_{m}^{(k+1)} v_{m}(x_{i_j})  
    =
 - 
 \left(
 \frac{\bar{u}_{(k+1)}(x^{(k)}_{i_j}) - \hat{u}_{(k)}(x^{(k)}_{i_j})}
     {\partial_x \hat{u}_{(k)}(x^{(k)}_{i_j})}
  \right)
  \quad
\text{ for } j = 1, ... , M.
    \label{eq:T_update_alpha}
\end{equation}

We will derive the update \cref{eq:s2} and discuss the choice of $\cQ \subset \cX$. For each interpolation particle $x^{(k)}_i \in \cX^{(k)}$, we employ the following update,
\beq
    x^{(k+1)}_i 
    = 
    x^{(k)}_i
    - \Delta t \chi_i^{(k)},
    \quad
    \chi_i^{(k)}
    :=
     \left(\frac{\partial_t \hat{u}(x^{(k)}_i,t_k)}
         {\partial_x \hat{u}(x^{(k)}_i,t_k)}\right),
    \quad
    i = 1, ..., N.
    \label{eq:particle_update}
\eeq
The term $\chi_{i}^{(k)}$ can be approximated by our time-update $\bar{u}_{(k+1)}(x)$ detailed in the previous step,
\begin{equation}
    \frac{\bar{u}_{(k+1)}(x^{(k)}_i) -    \hat{u}_{(k)}(x^{(k)}_i)}
       {\partial_x \hat{u}_{(k)}(x^{(k)}_i)}.
    \label{eq:T_update}
\end{equation}
Computing the transport mode coefficients of this term via empirical interpolation using $\cQ$ as interpolation points results in the formula \cref{eq:T_update_alpha}.

Next, we discuss how to choose $M$ points $\cQ \subset \cX$. We will assume $N > M$ and select a subset of $\cX$ to serve as interpolation points for $\{v_m\}_{m=1}^M$. This assumption is not restrictive, since in the case $N \le M$ similar procedures can be devised. We choose the subset that greedily maximizes the amplitude of $\partial_x \hat{u}_{(0)}$. Let $Q_j := \{x_{i_1}, ... , x_{i_j}\}$ and let $Q_0 := \emptyset$, then for $j = 1, ... , M$, 
\begin{equation}
i_{j}  := \argmax_{i} 
\{|\partial_x \hat{u}|(x_i,t_0): x_i \in \cX \setminus Q_{j-1} \}.
\label{eq:selectipts}
\end{equation}
The subset $\mathcal{Q} := Q_M = \{x_{i_j}\}_{j=1}^M$ are to be used with $\{v_m \}_{m=1}^M$.

\begin{remark}
The choice of $\cQ$ is critical and although we provided a heuristic procedure we use in our numerical examples, it is an open problem as to which points in the domain $\cQ \subset \Dom$ are optimal to choose. Note that the update \cref{eq:particle_update} is motivated by the characteristic curves to the problem \cref{eq:prob}.
\end{remark}

\begin{remark}[Efficient computations involving $\hat{T}_{(k)}$]
The inversion of $\hat{T}_{(k)}$ appearing in various forms above can be
performed efficiently, and a separate procedure is not necessary. For example,
consider the computation of the spatial deriative $\partial_x \hat{u}_{(k)}$ at
the point $x_i^{(k)}$ 
\[
  \partial_x \hat{u}_{(k)}(x_i^{(k)},t_k)
  =
  \sum_{n=1}^N \beta_n^{(k)} \hat{T}^\flat_{(k)} \zeta_n'(x_i^{(k)})
                             (\hat{T}_{(k)}^{-1})' (x_i^{(k)}).
\]
Now, $\hat{T}^\flat_{(k)}\zeta'_n(x_i^{(k)}) = \zeta'_n(x_i)$, and
$\hat{T}^{-1}_{(k)} (\hat{T}_{(k)} (x_i)) = x_i$ so $(\hat{T}_{(k)}^{-1})'
(\hat{T}_{(k)} (x_i))  = 1/\hat{T}_{(k)}' (x_i)$ where $\hat{T}'_{(k)} (x_i) =
\sum_{m=1}^M \alpha_m^{(k)} v'_m(x_i)$. Computing this only requires knowing
$v_m'(x_i)$. Moreover, when these computations are needed at points near
$\{x_i^{(k)}\}$, one uses the formula \cref{eq:eval}. 
\end{remark}

\begin{remark}[Invertibility of $\hat{T}_{(k)}$]
The injectivity of the transport map $\hat{T}_{(k)}$ \cref{eq:td-rm} is not
strictly enforced, although there are straightforward ways of doing so. For
example, in the representation of $\hat{T}_{(k)}$ in \cref{eq:td-rm}, one can
enforce that the derivative $\hat{T}_{(k)}'(x) = \sum_{m=1}^M \alpha_m^{(k)}
v_m'(x)$, is positive by checking this holds at the maxima of $|v_m'|$ for
$m=1, ... , M$.  This provides linear constraints on $\alpha_m^{(k)}$. However,
enforcing these constraints implies a modification of the dynamical system
induced by the explicit time-stepping scheme. This topic will be explored in a
separate work.  
\end{remark}

\subsubsection{Change of basis}
\label{sec:update_basis}

In the third and final step (S3), we close the time-update loop via a change of basis discussed in the \cref{sec:cob}, by computing $\hat{u}_{(k+1)}(x)$ from $\bar{u}_{(k+1)}(x)$, that is, $\{\beta_n^{(k+1)}\}_{n=1}^N$ from $\{\bar{\beta}_n^{(k+1)}\}_{n=1}^N$ in 
\begin{equation}
\begin{aligned}
\hat{u}_{(k+1)}(x_i^{(k+1)}) &= 
\sum_{n=1}^N \beta_n^{(k+1)} \hat{T}_{(k+1)}^\flat \zeta_n(x^{(k+1)}_i) \\
&=
\sum_{n=1}^N \beta_n^{(k+1)} \zeta_n(x_i) 
= 
\sum_{n=1}^N \bar{\beta}_n^{(k+1)} \hat{T}_{(k)}^\flat \zeta_n(x^{(k+1)}_i),
\quad i = 1, ..., N.
\end{aligned}
\tag{S3}
\label{eq:s3}
\end{equation}

We remark that the finite propagation speed of \cref{eq:prob} and the Courant-Friedrichs-Lewy condition for the full model implies that the interpolation particles $\{x^{(k+1)}_i\}_{i=1}^N$ should not move more than one full-model grid-width $\delta$ at each time-step, although this is not rigorously guaranteed for the online reduced model. 

\subsection{Reconstruction of the reduced solution}

As a result of the online computations, we obtain two sets of coefficients in \cref{eq:repr-Tpb}, namely $\{\alpha^{(k)}_m\}_{m=1}^M$ and $\{\beta^{(k)}_n\}_{n=1}^N$ for $k = 0,1, ... , K$. To utilize this reduced-model solution, one can reconstruct the solution on the full-model grid, and evaluate it to desired accuracy.  While this reconstruction procedure does depend on the full-model degree of freedom $N_\delta$, it can be computed in an embarrassingly parallel manner in both space and time: for each spatial point on the full-model grid and for each time.

\section{Algorithms}\label{sec:alg}

In this section, we provide a concise summary of the MATS algorithm introduced above. The offline computations are given in \cref{alg:offline} and the online computations in \cref{alg:online}.

\begin{algorithm}[h]
\caption{MATS Offline}
\label{alg:offline}
\begin{minipage}{0.95\textwidth} 
~\\
Compute basis $\{v_m\}_{m=1}^M$ of $\TTrb$
\begin{myitemize2}
    \item[1.] Collect snapshots $\cSg$ satisfying the signature condition 
    \item[2.] Compute DIP maps $\{T_\ell\}_{\ell=1}^\Ng$
    \item[3.] Compute low-rank transport modes $\{v_m\}_{m=1}^M$ from $\{T_\ell\}_{\ell=1}^\Ng$ via SVD
\end{myitemize2}
Compute basis $\{\zeta_n\}_{n=1}^N$ of $\VVrb$
\begin{myitemize2}
    \item[1.] Collect snapshots $\cSl$ from local time-parameter region 
    \item[2.] Compute reduced basis $\{\zeta_n\}_{n=1}^N$
    from $\cSl$ via SVD
    \item[3.] Compute interpolation points $\{x_i\}_{i=1}^N$ via EIM/DEIM
\end{myitemize2}
~
\end{minipage}
\end{algorithm}

\begin{algorithm}[h]
\caption{MATS Online}
\label{alg:online}
    \begin{minipage}{0.95\textwidth}
    Initialize $\{\alpha_m^{(0)}\}_{m=1}^M$,
    $\{\beta_n^{(0)}\}_{n=1}^N$, 
    $\hat{T}_{(0)} \gets \sum_{m=1}^M \alpha_m^{(0)} v_m,$
    $\{x_i^{(0)} \}_{i = 1}^N$,
    $\lambda \gets \Delta t / \Delta x$\\
    Time-step for $k=0, ... , K$
        \begin{myitemize2}
            \item[1.] Compute flux $\{\theta_n^{(k)}\}_{n=1}^N$, then compute flux update
            {\begin{align}
                \beta_n^{*}
                &\gets
                \beta_n^{(k)} + \lambda \theta_n^{(k)},
                \quad n = 1, ..., N.
                \tag{\ref{eq:s1-1}}
            \end{align}
            Compute source $\{\omega_n^{(k)}\}_{n=1}^N$, then compute source update
            \begin{align}
                \bar{\beta}_n^{(k+1)}
                &\gets
                \beta_n^{*} + \Delta t \omega_n^{(k)},
                \quad n = 1, ..., N.
                \tag{\ref{eq:s1-2}}
            \end{align}}
            \item[2.] Compute $\{\eta_m^{(k)}\}_{m=1}^M$, then compute transport map update
            {\begin{align}
                \alpha_m^{(k+1)}
                &\gets
                \alpha_m^{(k)}
                +
                \Delta t \eta_m^{(k)},
                \quad m = 1, ..., M,
                \tag{\ref{eq:s2}}
                \\
            \hat{T}_{(k+1)}
            &\gets
            \sum_{m=1}^M \alpha_m^{(k+1)} v_m,
            \notag
            \\
            x^{(k+1)}_i 
            &\gets
            \hat{T}_{(k+1)}(x_i),
            \quad i = 1, ..., N.
            \notag
            \end{align}}
            \item[3.] Change of basis
           \beq
               \beta_n^{(k+1)}
                \gets 
          \Call{Solve}{}
          \left[
          \begin{aligned}
\sum_{n=1}^N \beta_n^{(k+1)} \zeta_n(x_i) 
= 
\sum_{n=1}^N \bar{\beta}_n^{(k+1)} &\hat{T}_{(k)}^\flat \zeta_n(x^{(k+1)}_i)
        \\
        i&=1, ..., N.
        \end{aligned}\right]
        \tag{\ref{eq:s3}}
        \eeq
    \end{myitemize2}
        ~
    \end{minipage}
\end{algorithm}

\section{Numerical examples} \label{sec:numerics}

We will demonstrate the methods discussed in the previous sections through two
representative examples. One is the color equation with parametrized variable
speed, and the other is the Burgers' equation with a parametrized source term.
Throughout, we will denote by $N$ and $M$ the dimensions of our reduced model
\cref{eq:repr-Tpb}, $N = \dim \VVrb$, $M = \dim \TTrb.$ 

The runtime comparisons were performed on the Intel Xeon {\tt E5-2690v4}
2.6GHz CPU with 64GB of RAM, and all implementations were in \textsc{Python}.
The code used for the computations in this section is available in a public
repository \cite{matsrepo}.

\subsection{Color equation with parameter-dependent heterogeneous media}
\label{sec:color}
Let us consider the color equation with parametrized variable speed that is smooth,
\begin{equation}
    \left\{ \begin{aligned}
    \partial_t u + c(x;\bmu) \partial_x u &= 0
            \quad (x,t) \in (0,2) \times (0,1)\\
        u(x,0) &= u_0(x), \\
        u(0,t) &= u_0(0).
    \end{aligned} \right.
    \label{eq:prob-color}
\end{equation}
The initial condition is the cosine hump of width $0.4$ centered at $x= 0.25$,
\begin{equation}
    u_0(x) := 
    \begin{cases}
    \frac{1}{2} + \frac{1}{2} \cos(5\pi (x - 0.25)) 
        & \text{if } 0.05 \le x \le 0.45.
         \\
    0 & \text{if } x < 0.05 \text{ or } x > 0.45,
    \end{cases}
\end{equation}
This is a problem of the type \cref{eq:prob} with $f(u,x;\bmu) = c(x;\bmu) u$, $\psi(u,x;\bmu) = \partial_x c(x;\bmu) u$. 

The variable speed $c(\cdot;\bmu)$ is defined as a modulation of a constant background speed by a mixture of two parameter-dependent harmonic functions,
\beq
    c(x;\bmu) = 1.5 + \mu_1 \sin(\mu_2 x) + 0.1 \cos(\mu_3 x).
    \label{eq:cx}
\eeq
The parameters $\mu_2$ and $\mu_3$ correspond to fast and slow oscillations, respectively, and $\mu_1$ controls the amplitude of the fast oscillation. We choose the parameter domain $\cD := [0.25,0.5] \times [2\pi,6\pi] \times [\pi, 1.1\pi]$. One can numerically check that for this problem the signature condition (\cref{def:sgn}) is satisfied by the full model (see \cref{rmk:full}).

Our strategy for collecting snapshots is to (1) choose random samples for the parameter $\bmu \in \cD$, and (2) pick two time-intervals, one local (near initial time) and one global. Then the full-model solutions are computed for randomly drawn $\bmu$: Solutions for the local time-interval will form the local snapshots, and solutions for the global time-interval will form the the global snapshots.
    
We pick random parameters $\bmu$ by drawing 25 independent uniformly
distributed random samples in the domain $\cD$. We denote the set of these
samples to be $\cNl$, see \cref{tab:testerror-color} for a scatter plot. Local
snapshots computed with the full model are taken from equally spaced times in
$[0,0.08]$. Then the local snapshots are taken for time-parameter values
\[
    \cHl := \{(t,\bmu): t \in \cTl, \bmu \in \cNl \},
    \quad
    \cTl := \{0.02i: i = 0,1, ... ,4\}.
\]
Hence, we obtain the local snapshots $\cSl$ \cref{eq:sns_loc} of size $|\cSl| = 125$.

Global snapshots $\cSg$ are taken over longer periods of time,
\[
    \cHg := 
    \cHl \cup \{(t,\bmu): t \in \cTg, \bmu \in \cNl \},
    \quad
    \cTg := \{0.1i: i = 1,2, ... ,10\},
\]
yielding the global snapshots $\cSg$ \cref{eq:T_snapshot} of size $|\cSg| = 375$.

For the reduced model, we choose a uniform time-step size with a fixed ratio $\lambda = \Delta t / \Delta x = 0.5$. The reduced model is run up to $K=2400$ time-steps. 

Speedup results shown in \cref{fig:speedup} reflect the theoretical complexity
of $\cO(N)$ per time-step, where the runtime is independent of the size of the
full model. This is in contrast to the $\cO(N_\delta)$ of the full-model
solution, as well as the reduced model employing POD-DEIM that achieves
comparable accuracy. The offline computation time not including the full model
runtime was 136.07s.

We draw 10 test parameters from a uniform random distribution on $\cD$ (see \cref{tab:testerror-color}) and compare the reconstruction with the full model. The average $L^1$ relative error between the reduced model and the full model was computed every 100 time-steps for each test parameter. The result for individual test parameters for dimensions $(N,M) = (12,4)$  is shown in \cref{tab:testerror-color}: The error over a single time-step is below $10^{-3}$, but the error accumulates over time and can reach $10^{-2}$. Next, we vary the dimensions $(N,M)$ and compute the average error over the 10 test parameter samples. The result is shown in \cref{fig:speedup} and serves as numerical evidence that the $(N,M)$-width for the problem \cref{eq:prob-color} is small. The general trend is that for higher $(N,M)$ the error decreases, however this trend is not strictly monotone, especially with respect to $N$. 

A plot of the reduced and full-model solutions, together with the speed $c(x;\bmu)$ and the trajectory of the particles are shown in \cref{fig:romsol-color}, demonstrating good agreement between the reduced model and the full model.

\begin{remark}
The nonlinear reduced representation is mesh-free and the time-stepping reduced
solution is able to leave the domain as long as the problem can be extended
beyond the prescribed domain. This is the case in this example, since the
velocity field $c(\cdot ;\boldsymbol{\mu})$ \cref{eq:cx} yields such an
extension; so the reduced solution can be solved up to a final time longer than
that specified in the problem. 
\end{remark}

\begin{figure}
\centering
\begin{tabular}{cc}
\includegraphics[width=0.4\textwidth]{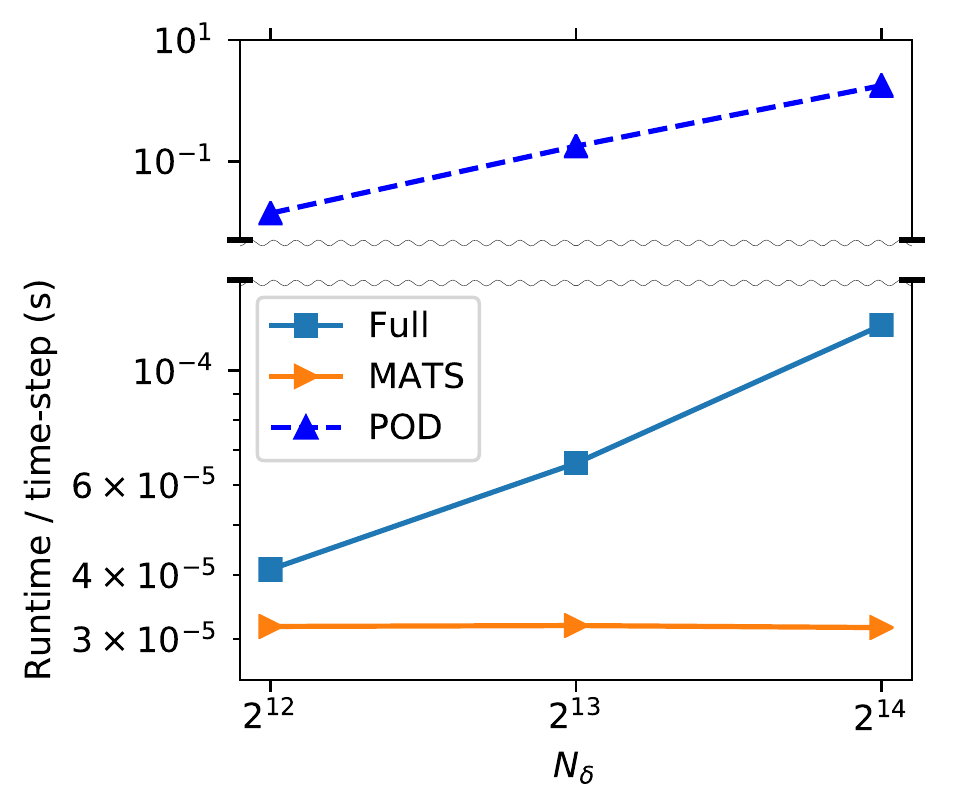}
&
\includegraphics[width=0.5\textwidth]{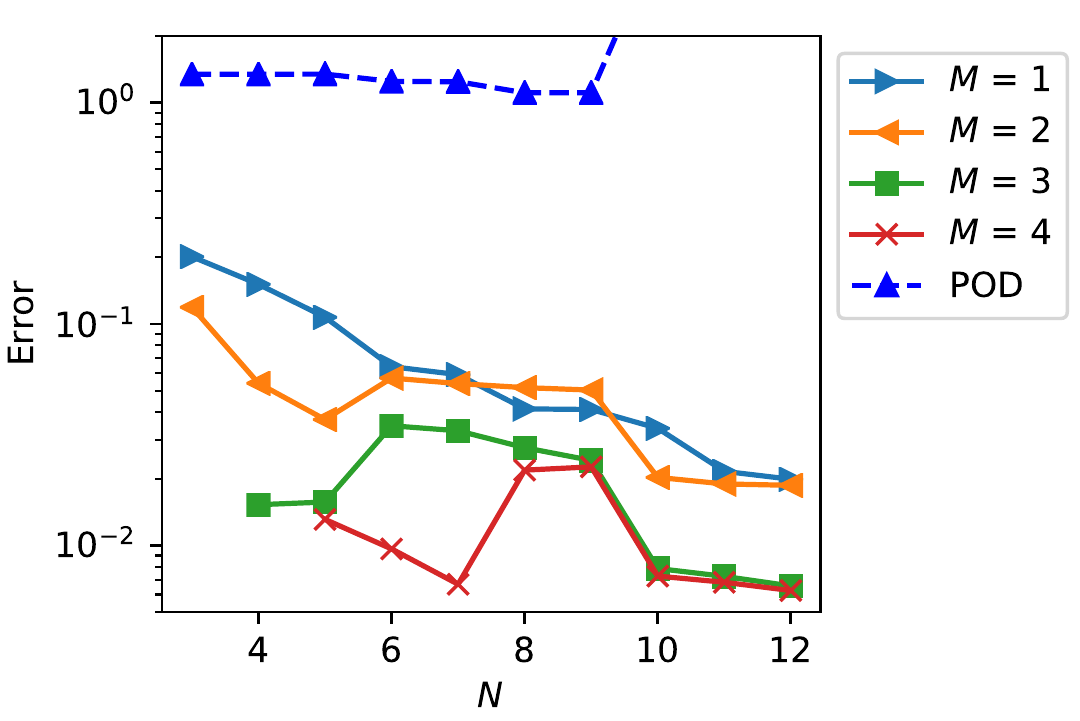}
\end{tabular}
\caption{Runtime comparison for single time-step of full model of varying size $N_\delta$ and corresponding reduced model of dimensions $(N,M) = (12,4)$, and POD-DEIM with comparable accuracy (left).  Average relative error plot for the color equation example \cref{eq:prob-color} with various dimensions $(N,M)$ (right).}
\label{fig:speedup}
\end{figure}

\begin{figure}
\centering
\begin{tabular}{cc}
    \includegraphics[height=0.45\textwidth,trim=1.0cm 0cm 0cm 0cm]{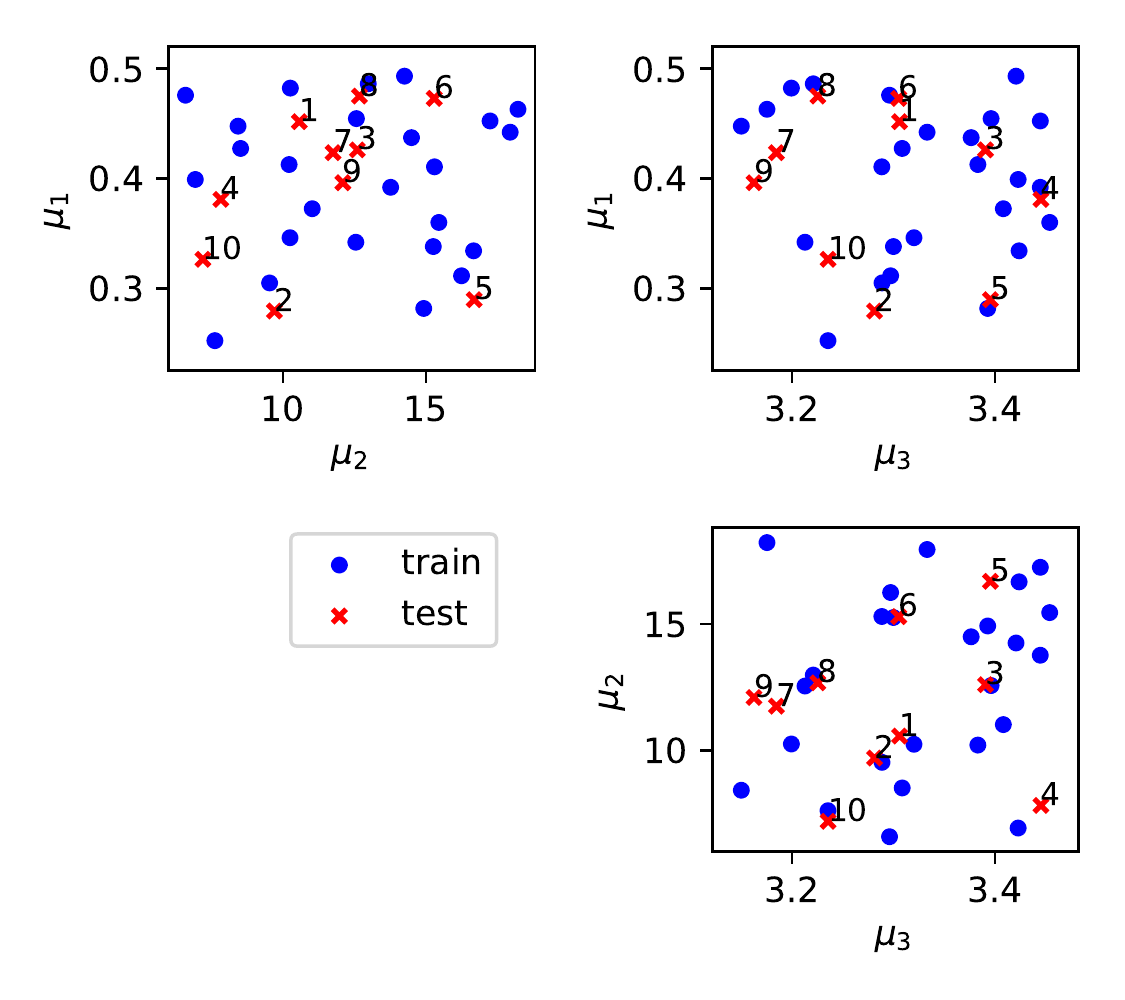}
     &  
     \includegraphics[height=0.46\textwidth,trim=1.0cm 0cm 0cm 0cm]{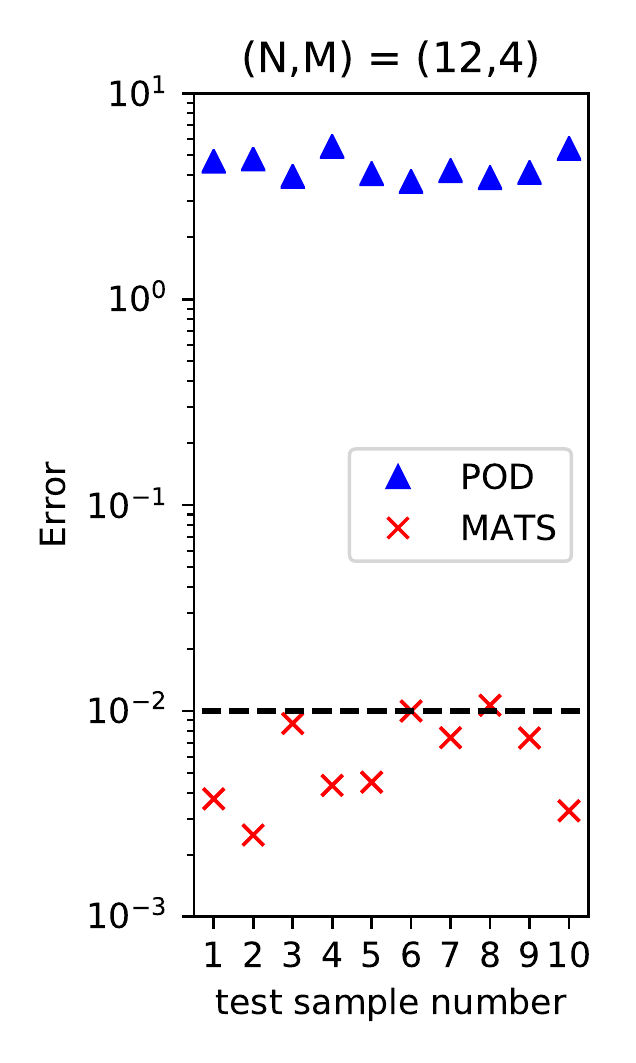}
     \\
\end{tabular}
\caption{Training and test parameter samples used for the color equation example (left). $L^1$-relative error of the reduced model of dimensions $(N,M) = (12,4)$ averaged over time, and corresponding error for POD-DEIM with dimension $N$ (right).}
\label{tab:testerror-color}
\end{figure}

\begin{figure}
\centering
\includegraphics[width=1.0\textwidth,trim=0cm 1cm 0cm 0cm]{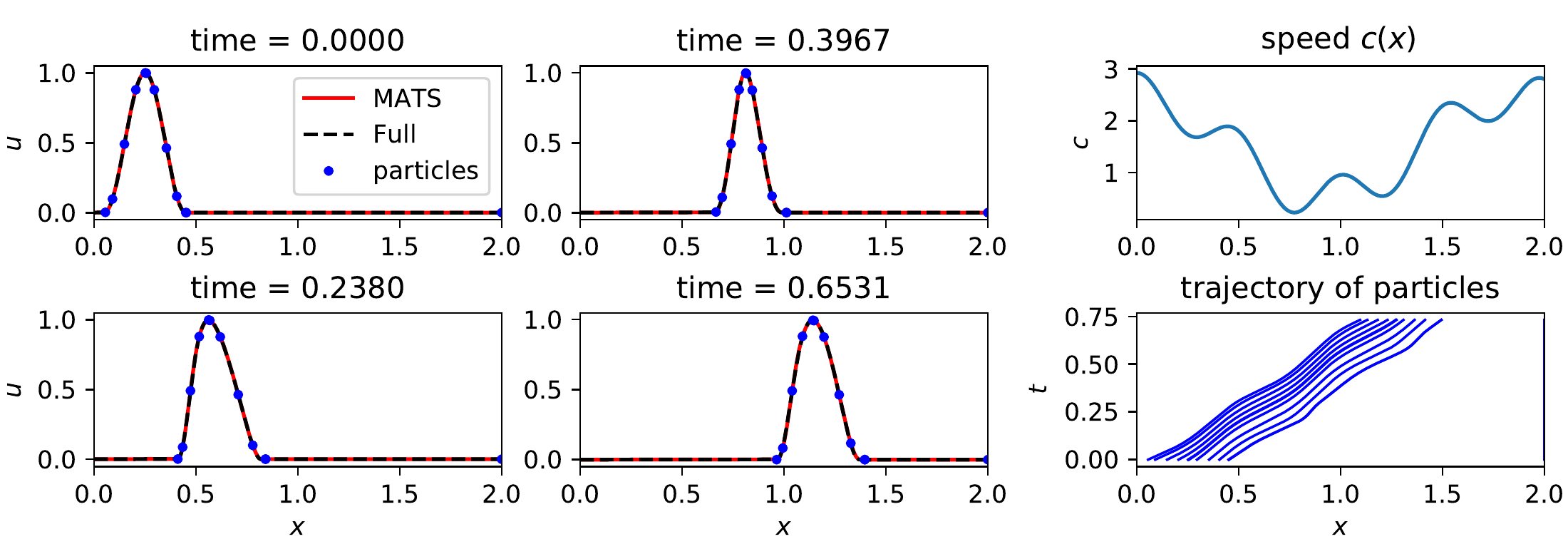}
\caption{Solution plot for the color equation example \cref{eq:prob-color}. The reduced-model solution $\hat{u}$, full-model solution $\uh$ and speed $c(x;\bmu)$ for the test parameter values $\mu_1 = \fixsp{0.43}$, $\mu_2=\fixsp{12.60}$, $\mu_3= \fixsp{3.39}$ corresponding to test parameter number 3 in \cref{tab:testerror-color}.  The dimensions of this reduced model is $(N,M) = (12,4)$.}
\label{fig:romsol-color}
\end{figure}

\subsection{Burgers' equation with reaction source} \label{sec:burgers}

Let us consider a nonlinear example, a Burgers' equation coupled with a
reactive equation. This is a model equation for the combustion problem where
chemical reaction occurs inside fluid flow. For fast reactions, computing
an accurate solution with standard methods is a computationally expensive task
\cite{Bourlioux91,leveque-yee}.
    \begin{equation}
        \left\{ \begin{aligned}
            \partial_t u + \partial_x \left( \frac{1}{2} u^2 \right) &= \mu_1 u(1-u)(u-\mu_2),
                \quad (x,t) \in (-5,5) \times (0, t_S(\bmu))\\
            u(x,0) &= u_0(x), \\
            u(0,t) &= u_0(0).
        \end{aligned} \right.
        \label{eq:prob-burgers}
    \end{equation}
The first parameter $\mu_1$ is the reaction coefficient, which determines
the time-scale of the reaction, the second parameter $\mu_2 \in (0,1)$
determines the unstable equilibrium, and $t_S(\bmu)$ denotes the shock
formation time.

We let the initial condition $u_0(x)$ be a sine-slope decreasing from 1 to 0,
\begin{equation}
    u_0(x) = 
    \begin{cases}
      1 & \text{if } x < -2, \\
        \frac{1}{2} - \frac{1}{2} \sin(\frac{\pi}{4} x) 
        & \text{if }-2 \le x \le 2, \\
      0 & \text{if } x > 2.
    \end{cases}
\end{equation}

For the full model, small mesh-size is necessary due to the time-step restriction $\mu_1 \Delta t \ll 1$, and since the spatial and temporal resolutions are related by the ratio $\lambda = \Delta t / \Delta x$ away from zero, small grid-size is required when the parameter $\mu_1$ is large (see \cite{leveque-yee}).

The strategy for collecting training samples is the same as in the previous example: We run the full model for various parameter values, then collect solutions which belong to a local time-interval close to the initial time, and these early solutions form the local snapshots $\cSl$. Then we widen the time-interval and collect solutions for the global snapshots $\cSg$.

For this example, we perform two experiments. In the individual experiments, the coefficient $\mu_1$ will be set to relatively lower and higher range corresponding to slow and fast reaction regimes. We distinguish the two ranges of parameters by the superscripts \textrm{S} and \textrm{F}. The parameters will be in the respective parameter domains,
\begin{equation}
 \cD^\textrm{S} 
    := 
    [ 50, 60] \times [0.1, 0.9], \quad
 \cD^\textrm{F} 
    := 
     [100,150]
    \times [0.1, 0.9].
    \label{eq:param-range-burgers}
\end{equation}
    
The parameters for the snapshots are 25 samples drawn from uniformly distributed random samples in their respective domains $\cD^\textrm{S},\cD^\textrm{F}$; we denote the set of these samples to be $\cNlslow,\cNlfast$. Since the reaction rate $\mu_1$ controls the time-scale, we collect the full-model solution during equally spaced times near the initial time $t = 0$ scaled by this rate,
\beq
    \cTlslow(\bmu) := \left\{\frac{0.1}{\mu_1} i: i = 0,1, ... ,4 \right\},
    \,\,
    \cTlfast(\bmu) := \left\{\frac{0.01}{\mu_1} i: i = 0,1, ... ,4 \right\}.
\eeq
The local snapshots are taken at the time-parameter values,
\beq
    \cHlslow := \{(t,\bmu): t \in \cTlslow(\bmu), \bmu \in \cNlslow \}, 
    \quad
    \cHlfast := \{(t,\bmu): t \in \cTlfast(\bmu), \bmu \in \cNlfast \}, 
\eeq
so we obtain $|\cSlslow| = |\cSlfast| = 125$ local snapshots.
    
Global snapshots $\cSgslow,\cSgfast$ \cref{eq:T_snapshot} are taken over longer time-intervals,
\begin{equation}
    \begin{aligned}
    \cHgslow := 
    \cHlslow \cup &\{(t,\bmu): t \in \cTgslow(\bmu), \bmu \in \cNlslow \}, \\
    \cHgfast := 
    \cHlfast \cup &\{(t,\bmu): t \in \cTgfast(\bmu), \bmu \in \cNlfast \},
    \end{aligned}
\end{equation}
in which
\begin{equation}
    \begin{aligned}
    \cTgslow(\bmu) &:= \left\{\frac{0.5 + 1.95i}{\mu_1}: i=0, ... ,9 \right\},
    \\
    \cTgfast(\bmu) &:= \left\{\frac{0.05 + 5.55i}{\mu_1}: i=0, ... ,9\right\}.
    \end{aligned}
\end{equation}
So for each experiment $|\cSgslow| = |\cSgfast| = 375$ snapshots are collected.
 
The online reduced model is run up to the time-steps
\begin{equation}
    K^\text{S} := 200,
    \quad
    K^\text{F}(\bmu) := 
    \left\lfloor \frac{1}{\Delta t} \left( \frac{10.0}{\mu_1} \right) \right\rfloor.
\end{equation}
For the case with slow reactions, we set a fixed number of time-steps
$K^\text{S}$ well before the shock forms, whereas for the fast case, we
time-step up to the shock formation. We do not expect the method to perform
well close to the shock formation time in the fast case, so we set a stopping
criteria: If the particles become too close or if the ordering of the particles
change before $K^\textrm{F}(\bmu)$ is reached, we discard the result (see
\cref{fig:nmerror-burgers}). 

The offline computation time not including the full model runtime was
104.80s. Runtime speedup behavior of the reduced model is very similar to
the previous example \cref{sec:color}, see \cref{fig:speedup}. The average
$L^1$ relative error over 10 different test values, for various dimensions of
$(N,M)$ is shown in \cref{fig:nmerror-burgers}. Although the problem is
nonlinear, the error behavior for the slow reaction case is similar to that of
the linear example (\cref{fig:speedup}). In contrast, for the fast reaction
case, one observes an erratic behavior in the error with respect to the
dimensions. Note that there are combinations of $(N,M)$ that do not reach the
shock formation time due to the loss of monotonicity in the particles. However,
the method does provide an approximation with the relative error in the order
of $10^{-3}$ for certain choices of $(N,M)$ near the diagonal ($N=M$). As in
the previous example, the achieved error indicates that the $(N,M)$-width for
this problem is also small, a numerical verification that the solution has a
small $(N,M)$-width, despite the nonlinearity in both the flux and the source
terms.

In the slow reaction case, the time-scale for the shock formation is
commensurate with the time-scale of convergence to the equilibrium of the
nonlinear reaction term \cite[Chapter 17]{fvmbook}. Therefore, the location
of shock has little variability for the slow reaction case, hence POD can
perform reasonably there. In the fast reaction case, there is more variability
in the location of the shock, causing POD to become inefficient.
    
A plot of the reduced-model and the full-model solution for the fast reaction case is shown in \cref{fig:romsol-burgers}. 

\begin{figure}
\centering
\includegraphics[width=0.8\textwidth]{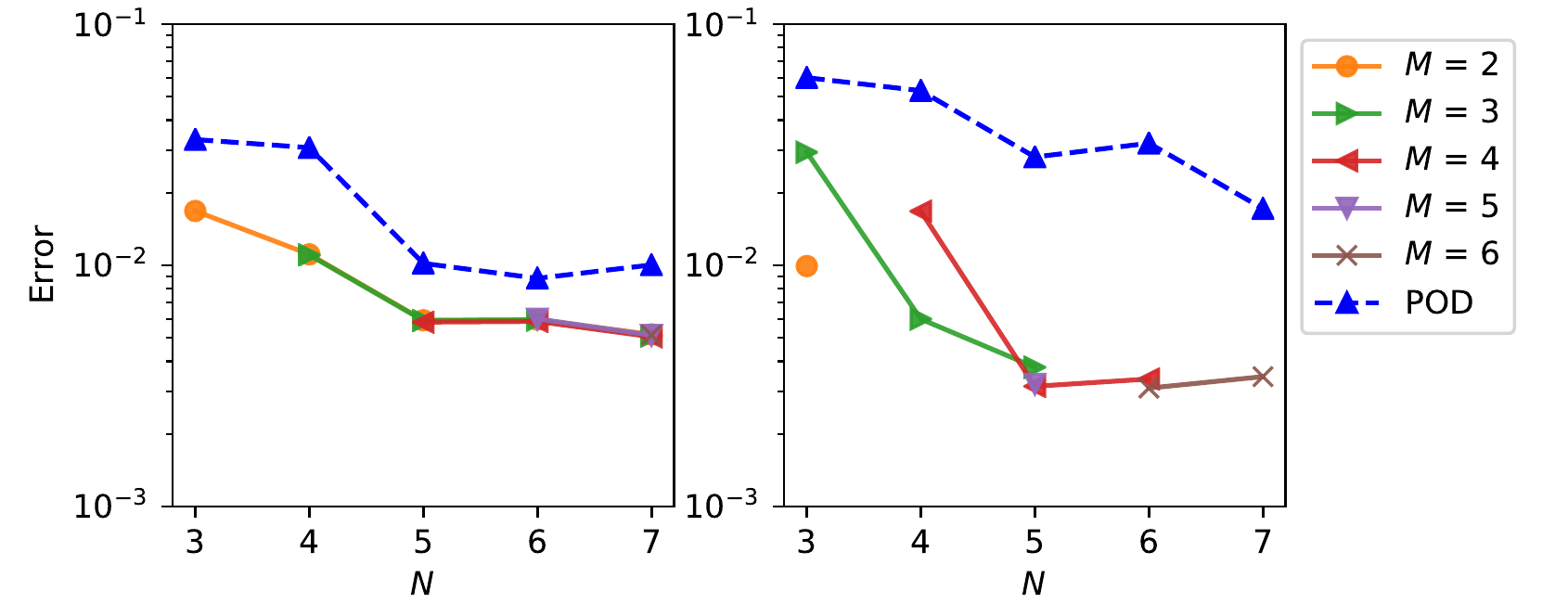}
\caption{Average relative error plot for the reactive Burgers' equation example \cref{eq:prob-burgers} for various dimensions $(N,M)$ for the slow reaction case (left) and the fast reaction case (right).}
\label{fig:nmerror-burgers}
\end{figure}

\begin{figure}
\centering
\includegraphics[width=0.7\textwidth]{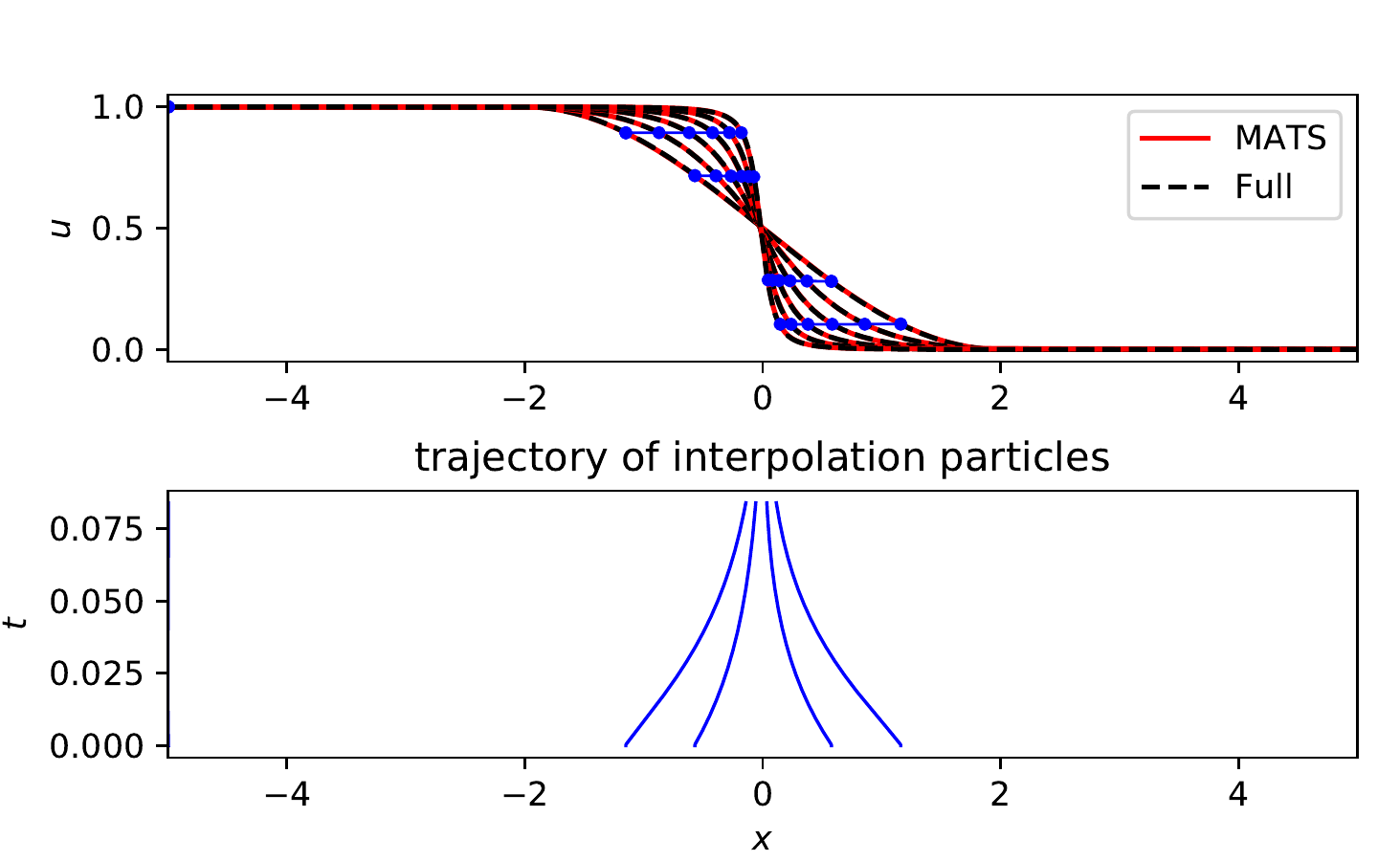}
\caption{Solution plot for the Burgers' equation example \cref{eq:prob-burgers}. The reduced-model solution $\hat{u}$ versus the full-model solution $\uh$ for the test parameter values $\mu_1=\fixsp{118.26}$, $\mu_2 =\fixsp{0.5240}$. The dimensions of this reduced model is $(N,M) = (5,4)$, and the solutions are shown at the times $t = \fixsp{0.0000}, \fixsp{0.0152}, \fixsp{0.0305}, \fixsp{0.0458}, \fixsp{0.0610}, \fixsp{0.0763} $.}
\label{fig:romsol-burgers}
\end{figure}

\subsection{Discussion}

We make some remarks regarding the  behavior of the error and shock formation for the Burgers' example. 

The error shown for both examples in \cref{fig:speedup,fig:nmerror-burgers} do not exhibit monotone decrease with respect to $N$ and $M$. Moreover, the accuracy does not appear to approach the levels of error for the full model. Similar behavior is also observed in the numerical examples of recent related works (see, e.g. \cite{P18AADEIM}). The likely contributing factors are (1) the low-rank Lagrangian frame causes the corresponding 1D mesh to be irregular, and the subspace $\VVrb(\bfeta)$ to be non-conformal with respect to the full-model discretization $\VVh$, (2) the first-order time-update with uniform time-step-size $\Delta t$ we employed in \cref{sec:rom} is a na\"ive projection of the first order full-model time-update which could yield less-than-optimal updates considering the projection to $\VVrb(t,\bmu)$, (3) the approximation error from the change of basis can accumulate over time, (4) better snapshot collection may be necessary for improved construction of $\VVrb$ and $\TTrb$. 

In the Burgers' example  shocks form, which implies that the characteristic curves merge. Such singular behavior in the characteristic curves is problematic for the current formulation of the reduced model. For example, in the update formula \cref{eq:T_update} the particles can switch their relative positions during time-stepping. This is a fairly standard phenomena for Lagrangian methods applied to the problem \cref{eq:prob} (see, e.g. \cite{mpm1}), but the case for particles $\{x_i^{(k)}\}_{i=1}^N$ has not yet been studied, to the best of our knowledge. Note that for approximations suggested in \cite{rim18b,rim19b} using DIP, monotonicity is preserved by construction and the merging characteristics are represented, allowing shock propagation for the reduced model.

\section{Conclusion and Outlook} \label{sec:conclusion}
Traditional model reduction methods based on linear approximations are
typically ineffective for systems governed by hyperbolic conservation laws
because the corresponding solution manifolds exhibit nonlinear transport
structures. The proposed approach MATS constructs nonlinear reduced models by
transporting subspaces along characteristic curves, which led to speedups of
orders of magnitude compared to traditional reduced models and full models in
the presented numerical examples.  This work highlighted two aspects of this
nonlinear reduction approach. First, the Kolmogorov $(N,M)$-width was
introduced, establishing a notion of best approximation of MATS from the
approximation theory point of view. Second, a computational procedure that
constructs reduced models based on MATS for physical systems governed by scalar
conservation laws was introduced. A crucial feature of the procedure is that
the governing equations of the conservation laws are time-stepped in the online
phase in contrast to, e.g., data-fit reduced models. Thus the proposed reduced
models are explicitly based on the physics (governing equations) of the systems
of interest, rather than being predictive in only a data-driven sense. 

The method introduced in this paper can potentially be generalized to
nonlinear systems of conservation laws in multiple spatial dimensions. The past
work of the first author \cite{rim18c} indicates extensibility to multiple
spatial dimensions; The first two authors have recently shown in \cite{Rim20b}
that shocks can be efficiently represented via a generalization of this reduced
model. But this extensibility itself is an open problem that we leave to future
works.

\clearpage

\section*{Acknowledgments}
The work of the first author (Rim) and second author (Peherstorfer) was partially supported by the Air Force Center of Excellence on Multi-Fidelity Modeling of Rocket Combustor Dynamics under Award Number FA9550-17-1-0195 and AFOSR MURI on multi-information sources of multi-physics systems under Award Number FA9550-15-1-0038 (Program Manager Dr.~Fariba Fahroo). The third author (Mandli) was partially supported by the NSF under Grant No. OAC-1735609 and DMS-1720288. The numerical experiments were computed with support through the NYU IT High Performance Computing resources, services, and staff expertise.

\clearpage

\appendix

\section{Supplementary details regarding \cref{expl:burgers_random,expl:wave}}
\label{sec:details}

\subsection{A sketch of proof for \cref{expl:burgers_random}}

Recall that the manifold contains discretized versions of white noise, of the
form $\sum_{n=1}^{N_\delta} b_n \varphi_n$ with the coefficients
$\{b_n\}_{n=1}^{N_\delta}$ drawn from independent normal distributions (where
$\{\varphi_n\}_{n=1}^{N_\delta}$ are the nodal basis functions
\cref{eq:fom_repr}).  Let $\Dom := (0,1)$, and let us draw $S$ snapshots from
the manifold. Then the inner products between the snapshots are given by
\beq
    \bfC_{\ell, \ell'}
    :=
    \left(
    \sum_{n=1}^{N_\delta} b_{n\ell} \varphi_n(x),
    \sum_{n=1}^{N_\delta} b_{n\ell'} \varphi_n(x)
    \right),
    \quad
    \text{ where } \ell, \ell' = 1, ... , S.
\eeq
As $N_\delta \to \infty$, by the law of large numbers $\bfC_{\ell,\ell'}$
converges to the identity $\bfI_{S}$ almost surely. Consequently, the
calculation of the $(N,M)$-width of the original random manifold is equivalent
to that of the finite-dimensional and deterministic manifold
\beq
    \hat{\cM}_\delta
    :=
    \left\{ 
        u \in \VV_\delta : 
        \Norm{u}{L^2(\Dom)} = 1 
    \right\}.
\eeq
This manifold has the Lebesgue covering dimension $N_\delta -1$, and one can
show that $\dim \cV(\VVrb, \TTrb) \le \dim(\VVrb) + \dim(\TTrb)$. This implies
that $\hat{\cM} \setminus \cV(\VVrb,\TTrb)$ has a large dimension (at the scale
of $N_\delta$) regardless of the choice of $\VVrb$ and $\TTrb$. Since all
members of $\hat{\cM}_\delta$ have unit $L^2$-norm, this implies the
$(N,M)$-width has the scale of $1$ for $N,M \ll N_\delta$.

\subsection{Remarks for \cref{expl:wave}}

Here we give a detailed explanation of how the $(N,M)$-width \cref{eq:nmwidth}
and $(N,M,L)$-width \cref{eq:nmlwidth} show distinct decay behavior for the
solution manifold of the 1D wave equation.

Let us consider the manifold containing the solutions to the wave equation with
a compactly supported initial condition $u_0$ and zero initial velocity
$u_t(x,0) = 0$. We then choose a final time small enough so that the support of
the solution does not interact with the boundary (alternatively, one may
choose aribtrary initial data and a radiation boundary condition). The
solution is given by the d'Alembert's formula
\beq
    u(x,t)
    =
    \half
    \left(
        u_0(x-t)
        +
        u_0(x+t)
    \right).
    \label{eq:wave_sol}
\eeq
Therefore, the solution $u(x,t)$ can be expressed as a superposition of two
approximations of the form \cref{eq:repr-Tpb}. (That is, we choose $\TTrb =
\Span_{\TT} \{1, \Id\}$ for this approximation). The solution manifold
$\cM_\delta$ will be given as the approximation $u_\delta(x,t)$ of $u(x,t)$ in
$\VV_\delta$ satisfying \cref{eq:fm_error}, and we choose $\VVrb = \Span
\{u_\delta(x,0)\}$. This constructively proves the claim that the manifold of
this PDE $\cM_\delta$ satisfies $d(1,2,2;\cM_\delta) = \cO(\delta^r)$. Below,
we refer to this approximation as the two-term approximation.

\begin{figure}
\centering
\includegraphics[width=0.9\textwidth]{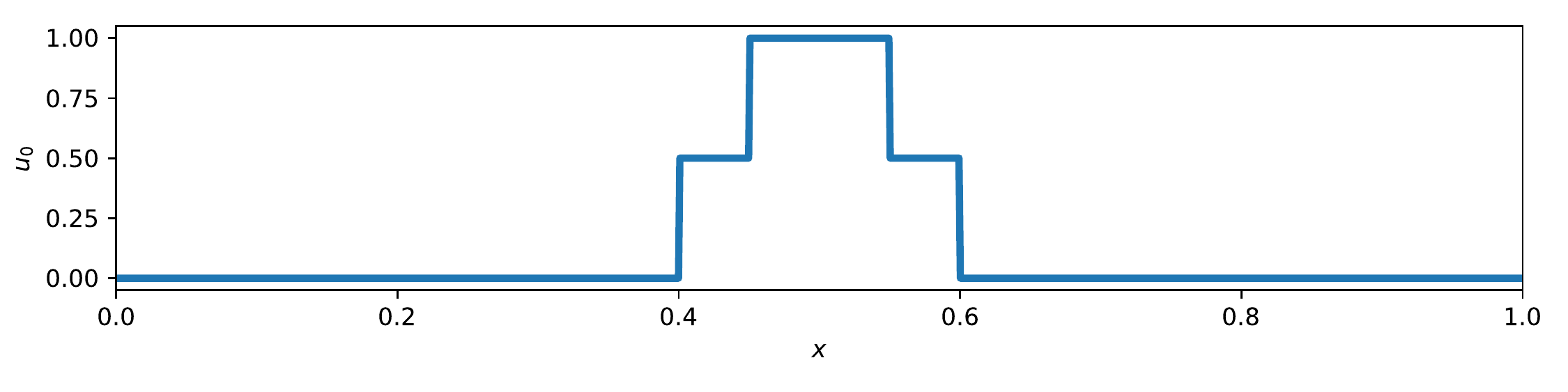}
\caption{The initial condition $u_0$ \cref{eq:expl_wave_u0}.}
\label{fig:expl_wave_u0}
\end{figure}

Next, we provide the rationale for expecting $d(N,2;\cM_\delta)$ to decay
slowly with respect to $N$. The argument here applies to any initial condition
with more than two jump discontinuities in the initial condition, but we shall
keep the discussion elementary by focusing on a specific piecewise constant
initial condition,
\beq
    u_0 (x)
    :=
    \begin{cases}
    1.0 & \text{ if } 0.45 \le x < 0.55,
    \\
    0.5 & \text{ if } 0.4 \le x < 0.45\text{ or } 0.55 \le x < 0.6
    \\
    0 & \text{ otherwise.}
    \end{cases}
    \label{eq:expl_wave_u0}
\eeq
A plot of $u_0$ is shown in \cref{fig:expl_wave_u0}. 

Consider using a single term of the form \cref{eq:repr-Tpb} to approximate
\cref{eq:wave_sol}. In contrast to the two-term case, a straightforward way to
exploit the traveling structure is available only if the supports of the
left-going and right-going waves do not overlap. A depiction is given in
\cref{fig:expl_wave_sep}. In this regime, the interval where the solution is
constant, e.g. $(0.5-\eps, 0.5+\eps)$ for a small $\eps > 0$, can be dilated to
produce a traveling wave. Set two transport modes, 
\beq
    v_1(x) = \Id(x),
    \quad
    v_2(x) = \begin{cases}
                \frac{1}{10\eps}(x - 0.5) & \text{ if } |x - 0.5| < \eps,\\
                0.1    & \text{ if } x > 0.5 + \eps, \\
               -0.1    & \text{ if } x < 0.5 - \eps, \\
             \end{cases}
\eeq
The two maps are illustrated on the right of \cref{fig:expl_wave_sep}. The
transport map constructed with $v_1$ and $v_2$ does not approximate the
characteristic curves, unlike in the two-term approximation. For concreteness,
one may fix $\TTrb = \Span_\TT \{v_1,v_2\}$ as the subset of diffeomorphisms,
despite the possibility this is not the optimal set of transport maps with
dimension two.

\begin{figure}
\centering
\includegraphics[width=0.95\textwidth]{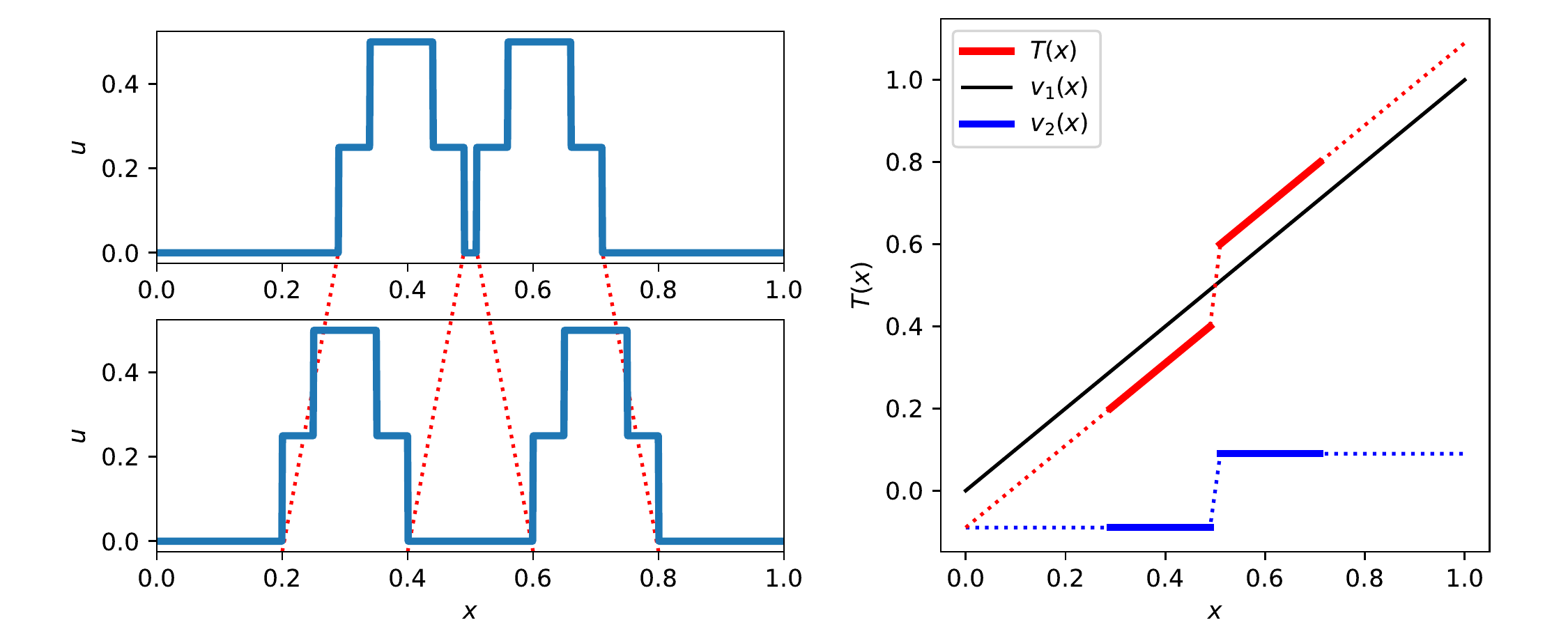}
\caption{The time-evolution of the left- and right-going waves (left
column), and an example of transport modes $v_1,v_2$, along with the transport
map $T \in \Span_\TT \{v_1, v_2\}$ (right column). The dotted red lines in the
left column correspond to two dashed red lines in the right figure.}
\label{fig:expl_wave_sep}
\end{figure}

The approximation task is non-trivial when the left- and right-going waves are
not separated by an interval of constancy. The solution of the wave equation
during these times are plotted in \cref{fig:expl_wave_vars}. For any function
$w \in \VV$, composition with the transport map $T \in \TTrb$ or its inverse
preserves the range of values, that is $w(\Dom) \supset w(T^{-1}(\Dom))$. In
our example, the range of values $u(\Dom,t)$ are point sets that changes
discontinuously over time, hence $\dim \VVrb > 1$ is required to avoid
pointwise $\cO(1)$ error. 

\begin{figure}
\centering
\includegraphics[width=0.95\textwidth]{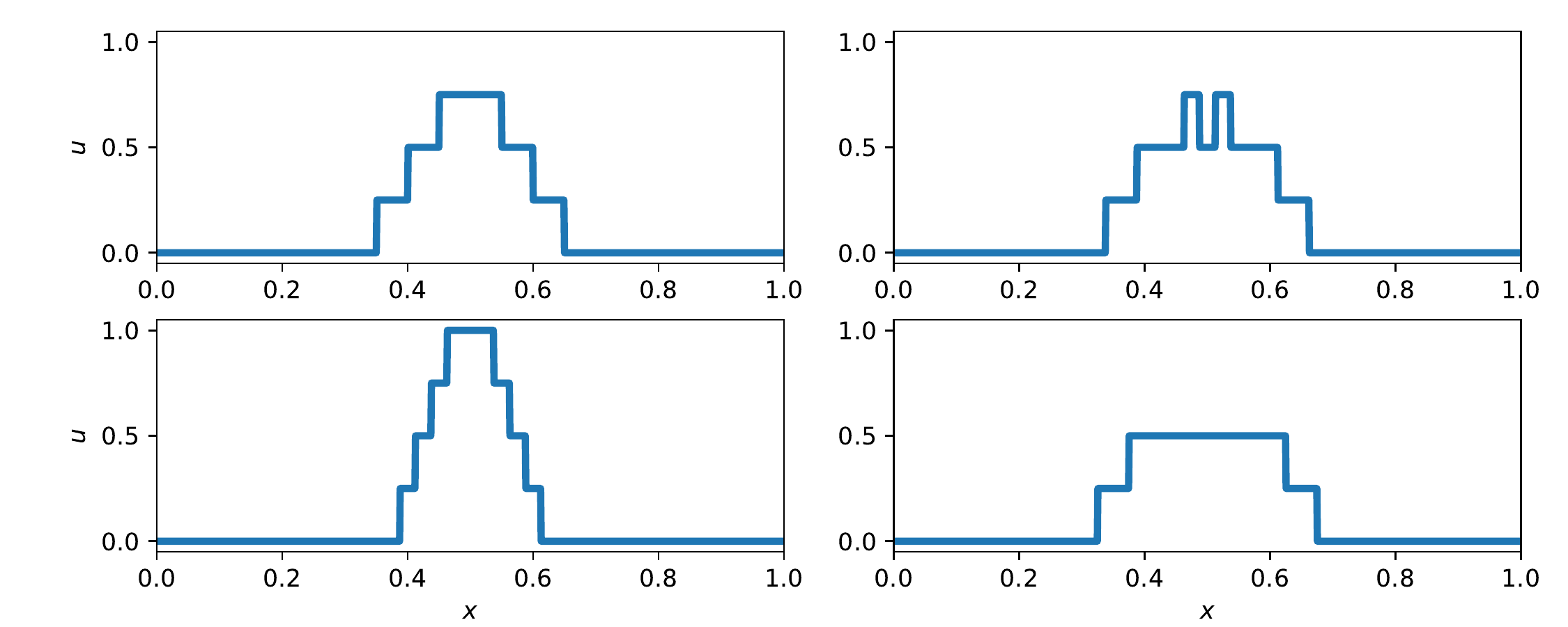}
\caption{The evolution of the wave equation with the initial condition $u_0$
\cref{eq:expl_wave_u0}.}
\label{fig:expl_wave_vars}
\end{figure}

Moreover, the transport maps in $\TTrb$ are not able to align the sharp
gradients of the evolving solution. This can be shown by comparing the number
of constraints necessary with $\dim \TTrb$. In this case, the standard
Kolmogorov $N$-width estimates for the advection equation
\cite{Ohlberger16,welper17} for a single jump applies locally, leading to a
slow decay in the error. Thus, we expect $d(N,2;\cM_\delta)$ to decay slowly. A
rigorous proof in a more general setting will require a more comprehensive
treatment and is outside the scope of this work.

\section{Transport modes of the advection equation} \label{appdx:expl_advec}

We give a particularly simple example of DIP maps to help illustrate the
concept. For another example, we refer the reader to \cite{rim18b}. 

Consider the advection equation $u_t + u_x = 0$ on $\Dom = (0,5)$ with the
initial condition $u(x,0) = u_0(x)$ in which $u_0$ is given by
\beq
    u_0(x) = 
    \begin{cases}
        \half  + \half \cos (2\pi (x-1)) & \text{ if } |x - 1| < \half,\\
        0 & \text{ otherwise.}
    \end{cases}
\eeq
The function $u_0$ is plotted in \cref{fig:advection_sig}. The solution has the
signature $(0, +1, -1, 0)$ before the solution interacts with the boundary, so
the solution satisfies the signature condition. The computed transport mode in
$\supp u_0$ is approximately the constant function. This corresponds to the
translation motion in the solution manifold.

\begin{figure}[ht!]
\centering
\includegraphics[width=0.95\textwidth]{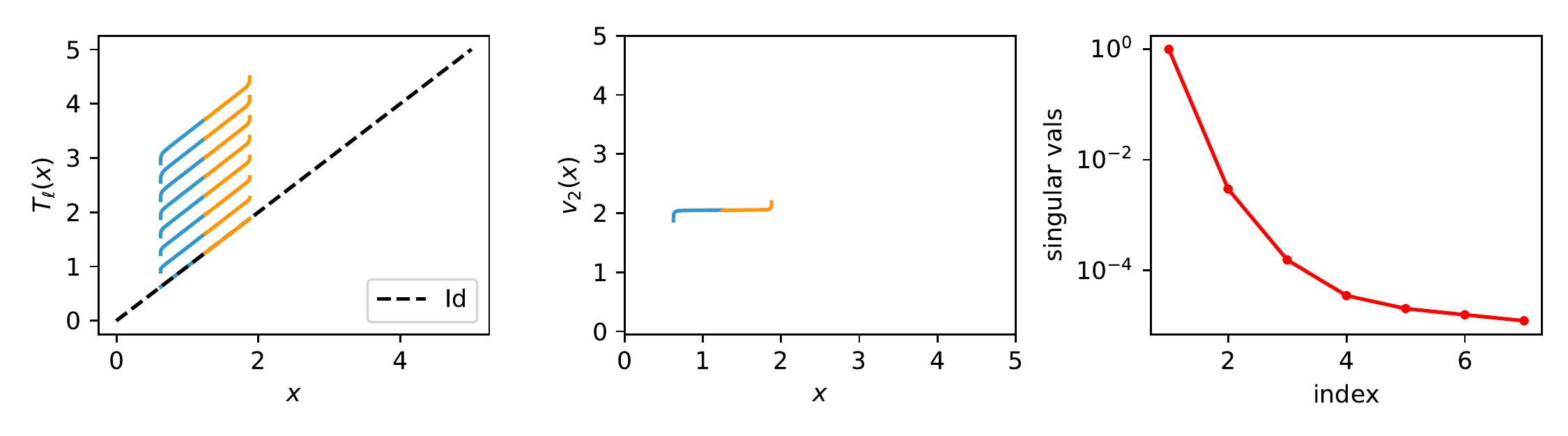}
\caption{The DIP maps $T_\ell(x)$ (left), the computed transport mode
$v_2(x)$ (middle) and the singular values of the matrix $\bfC$ \cref{eq:Cll}
(right).}
\label{fig:expl_advec_tmodes}
\end{figure}

\end{document}